\documentclass[12pt]{amsart}
\usepackage{amsmath}
\usepackage{amsfonts}
\usepackage{amssymb}
\usepackage[all,cmtip]{xy}           
\usepackage{mathrsfs}
\usepackage{bbm}
\usepackage{bbding}
\usepackage{txfonts}
\usepackage[shortlabels]{enumitem}
\usepackage{ifpdf}
\usepackage{extarrows}
\ifpdf
\usepackage[colorlinks,final,backref=page,hyperindex]{hyperref}
\else
\usepackage[colorlinks,final,backref=page,hyperindex,hypertex]{hyperref}
\fi
\usepackage{amscd}
\usepackage{tikz-cd}
\usepackage[active]{srcltx}

\makeatletter

\newtheorem{theorem}{Theorem}[section]
\newtheorem{prop}[theorem]{Proposition}
\newtheorem{lemma}[theorem]{Lemma}
\newtheorem{coro}[theorem]{Corollary}
\newtheorem{prop-def}{Proposition-Definition}[section]

\theoremstyle{definition}
\newtheorem{defn}[theorem]{Definition}

\newtheorem{remark}[theorem]{Remark}
\newtheorem{exam}[theorem]{Example}

\newcommand{\nc}{\newcommand}

\nc{\delete}[1]{{}}
\nc{\mmargin}[1]{}

\nc{\mlabel}[1]{\label{#1}}  

\nc{\mref}[1]{\ref{#1}}  
\nc{\mbibitem}[1]{\bibitem{#1}} 

\delete{
	\nc{\mlabel}[1]{\label{#1}  
		{\hfill \hspace{1cm}{\bf{{\ }\hfill(#1)}}}}
	\nc{\cite}[1]{\cite{#1}{{\bf{{\ }(#1)}}}}  
	\nc{\mref}[1]{\ref{#1}{{\bf{{\ }(#1)}}}}  
	\nc{\mbibitem}[1]{\bibitem[\bf #1]{#1}} 
}

 \font\cyrs=wncyr7

\newcommand{\bk}{{\mathbf{k}}}

\nc{\vep}{\varepsilon}
\nc{\bin}[2]{ (_{\stackrel{\scs{#1}}{\scs{#2}}})}  
\nc{\binc}[2]{(\!\! \begin{array}{c} \scs{#1}\\
		\scs{#2} \end{array}\!\!)}  
\nc{\bincc}[2]{  ( {\scs{#1} \atop
		\vspace{-1cm}\scs{#2}} )}  
\nc{\oline}[1]{\overline{#1}}
\nc{\mapm}[1]{\lfloor\!|{#1}|\!\rfloor}
\nc{\bs}{\bar{S}}
\nc{\la}{\longrightarrow}
\nc{\ot}{\otimes}
\nc{\rar}{\rightarrow}
\nc{\lon }{\,\rightarrow\,}
\nc{\dar}{\downarrow}
\nc{\dap}[1]{\downarrow \rlap{$\scriptstyle{#1}$}}
\nc{\defeq}{\stackrel{\rm def}{=}}
\nc{\dis}[1]{\displaystyle{#1}}
\nc{\dotcup}{\ \displaystyle{\bigcup^\bullet}\ }
\nc{\Gr}{\mathrm{Gr}}
\nc{\hcm}{\ \hat{,}\ }
\nc{\hts}{\hat{\otimes}}
\nc{\hcirc}{\hat{\circ}}
\nc{\lleft}{[}
\nc{\lright}{]}
\nc{\curlyl}{\left \{ \begin{array}{c} {} \\ {} \end{array}
	\right .  \!\!\!\!\!\!\!}
\nc{\curlyr}{ \!\!\!\!\!\!\!
	\left . \begin{array}{c} {} \\ {} \end{array}
	\right \} }
\nc{\longmid}{\left | \begin{array}{c} {} \\ {} \end{array}
	\right . \!\!\!\!\!\!\!}
\nc{\ora}[1]{\stackrel{#1}{\rar}}
\nc{\ola}[1]{\stackrel{#1}{\la}}
\nc{\scs}[1]{\scriptstyle{#1}} \nc{\mrm}[1]{{\rm #1}}
\nc{\dirlim}{\displaystyle{\lim_{\longrightarrow}}\,}
\nc{\invlim}{\displaystyle{\lim_{\longleftarrow}}\,}
\nc{\dislim}[1]{\displaystyle{\lim_{#1}}} \nc{\colim}{\mrm{colim}}
\nc{\mvp}{\vspace{0.3cm}} \nc{\tk}{^{(k)}} \nc{\tp}{^\prime}
\nc{\ttp}{^{\prime\prime}} \nc{\svp}{\vspace{2cm}}
\nc{\vp}{\vspace{8cm}}
\nc{\modg}[1]{\!<\!\!{#1}\!\!>}
\nc{\intg}[1]{F_C(#1)}
\nc{\lmodg}{\!<\!\!}
\nc{\rmodg}{\!\!>\!}
\nc{\cpi}{\widehat{\Pi}}
\nc{\ssha}{{\mbox{\cyrs X}}} 
\nc{\tsha}{{\mbox{\cyrt X}}}
\nc{\shpr}{\diamond}    
\nc{\labs}{\mid\!}
\nc{\rabs}{\!\mid}

\nc{\C}{{\mathrm{C}}}
 \nc{\dd}{{\mathrm{d}}}

\nc{\ad}{\mrm{ad}}
\nc{\ann}{\mrm{ann}}
\nc{\Aut}{\mrm{Aut}}
\nc{\DA}{{\mathsf{DL}_\lambda}}
\nc{\Alg}{{\mathrm{Lie}}}
\nc{\DO}{{\mathsf{DO}_\lambda}}
\nc{\bim}{\mbox{-}\mathsf{Rep}}
\nc{\md}{\mbox{-}\mathsf{rep}}
\nc{\br}{\mrm{bre}}
\nc{\can}{\mrm{can}}
\nc{\rchar}{\mrm{char}}
\nc{\cok}{\mrm{coker}}
\nc{\de}{\mrm{dep}}
\nc{\dtf}{{R-{\rm tf}}}
\nc{\dtor}{{R-{\rm tor}}}

\nc{\Div}{{\mrm Div}}
\nc{\Diff}{\mrm{DL}}
\nc{\Diffl}{\mathsf{DL}_\lambda}
\nc{\diffo}{{\mathsf{DO}_\lambda}}
\nc{\Dif}{{\mathfrak{Dif}^\lambda}}
\nc{\Difinfty}{{\mathfrak{Dif}^\lambda_\infty}}
\nc{\alg}{\mathsf{Lie}}
\nc{\End}{\mrm{End}}
\nc{\Ext}{\mrm{Ext}}
\nc{\Fil}{\mrm{Fil}}
\nc{\Fr}{\mrm{Fr}}
\nc{\Frob}{\mrm{Frob}}
\nc{\Gal}{\mrm{Gal}}
\nc{\GL}{\mrm{GL}}
\nc{\Hom}{\mrm{Hom}}
\nc{\Hoch}{\mrm{Hoch}}
\nc{\hsr}{\mrm{H}}
\nc{\hpol}{\mrm{HP}}
\nc{\id}{\mrm{id}}
\nc{\im}{\mrm{im}}
\nc{\Id}{\mrm{Id}}
\nc{\ID}{\mrm{ID}}
\nc{\Irr}{\mrm{Irr}}
\nc{\incl}{\mrm{incl}}
\nc{\Ker}{\mrm{Ker}}
\nc{\length}{\mrm{length}}
\nc{\NLSW}{\mrm{NLSW}}
\nc{\Lie}{\mrm{Lie}}
\nc{\mchar}{\rm char}
\nc{\NR}{{\rm NR}}
\nc{\mpart}{\mrm{part}}
\nc{\ql}{{\QQ_\ell}}
\nc{\qp}{{\QQ_p}}
\nc{\rank}{\mrm{rank}}
\nc{\rcot}{\mrm{cot}}
\nc{\rdef}{\mrm{def}}
\nc{\rdiv}{{\rm div}}
\nc{\rtf}{{\rm tf}}
\nc{\rtor}{{\rm tor}}
\nc{\res}{\mrm{res}}
\nc{\Sh}{{\mathrm{Sh}}}
\nc{\SL}{\mrm{SL}}
\nc{\Spec}{\mrm{Spec}}
\nc{\sgn}{{\mathrm{sgn}}}
\nc{\tor}{\mrm{tor}}
\nc{\Tr}{\mrm{Tr}}
\nc{\tr}{\mrm{tr}}
\nc{\wt}{\mrm{wt}}
\nc{\op}{\mrm{op}}
\nc{\rmH}{ {\mathrm{H}}}
\nc{\rmC}{ {\mathrm{C}}}

\nc{\bfk}{{\bf k}}
\nc{\bfone}{{\bf 1}}
\nc{\bfzero}{{\bf 0}}
\nc{\detail}{\marginpar{\bf More detail}
	\noindent{\bf Need more detail!}
	\svp}
\nc{\gap}{\marginpar{\bf Incomplete}\noindent{\bf Incomplete!!}
	\svp}
\nc{\FMod}{\mathbf{FMod}}
\nc{\Int}{\mathbf{Int}}
\nc{\Mon}{\mathbf{Mon}}
\nc{\remarks}{\noindent{\bf Remarks: }}
\nc{\Rep}{\mathbf{Rep}}
\nc{\Rings}{\mathbf{Rings}}
\nc{\Sets}{\mathbf{Sets}}
\nc{\ob}{\mathsf{Ob}}

\nc{\BA}{{\mathbb A}}   \nc{\CC}{{\mathbb C}}
\nc{\DD}{{\mathbb D}}   \nc{\EE}{{\mathbb E}}
\nc{\FF}{{\mathbb F}}   \nc{\GG}{{\mathbb G}}
\nc{\HH}{{\mathbb H}}   \nc{\LL}{{\mathbb L}}
\nc{\NN}{{\mathbb N}}   \nc{\PP}{{\mathbb P}}
\nc{\QQ}{{\mathbb Q}}   \nc{\RR}{{\mathbb R}}
\nc{\TT}{{\mathbb T}}   \nc{\VV}{{\mathbb V}}
\nc{\ZZ}{{\mathbb Z}}   \nc{\TP}{\widetilde{P}}
\nc{\m}{{\mathbbm m}}

\nc{\cala}{{\mathcal A}}    \nc{\calc}{{\mathcal C}}
\nc{\cald}{\mathcal{D}}     \nc{\cale}{{\mathcal E}}
\nc{\calf}{{\mathcal F}}    \nc{\calg}{{\mathcal G}}
\nc{\calh}{{\mathcal H}}    \nc{\cali}{{\mathcal I}}
\nc{\call}{{\mathcal L}}    \nc{\calm}{{\mathcal M}}
\nc{\caln}{{\mathcal N}}    \nc{\calo}{{\mathcal O}}
\nc{\calp}{{\mathcal P}}    \nc{\calr}{{\mathcal R}}
\nc{\cals}{{\mathcal S}}    \nc{\calt}{{\Omega}}
\nc{\calv}{{\mathcal V}}    \nc{\calw}{{\mathcal W}}
\nc{\calx}{{\mathcal X}}

\nc{\fraka}{{\mathfrak A}}
\nc{\frakb}{\mathfrak{b}}
\nc{\frakc}{{\frak C}_\mathrm{Lie}}
\nc{\frakg}{{\frak g}}
\nc{\frakL}{{\frak L}}
\nc{\frakl}{{\frak l}}
\nc{\fraks}{{\frak s}}
\nc{\frakB}{{\frak B}}
\nc{\frakm}{{\frak M}}
\nc{\frakM}{{\frak M}}
\nc{\frakp}{{\frak p}}
\nc{\frakW}{{\frak W}}
\nc{\frakX}{{\frak X}}
\nc{\frakS}{{\frak S}}
\nc{\frakA}{{\frak A}}
\nc{\frakx}{{\frakx}}
\nc{\frakC}{{\frak{C}}}
\nc{\frakh}{{\frak h}}

\nc{\yj}[1]{\textcolor{green}{#1  }}

\begin{document}

\title [Differential Lie algebras]{Formal deformations, cohomology theory  and $L_\infty[1]$-structures for      differential Lie algebras  of arbitrary weight}
\author{Weiguo Lyu, Zihao Qi, Jian Yang and Guodong  Zhou}

\address{Weiguo Lyu\\School of Mathematical Sciences, Anhui University, Hefei 230601, China}
  \email{wglyu@ahu.edu.cn}

\address{Zihao Qi\\  School of Mathematical Sciences,
	Fudan University, Shanghai 200433, China}
  \email{qizihao@foxmail.com}

\address{Jian Yang and Guodong  Zhou\\
School of Mathematical Sciences,  Key Laboratory of Mathematics and Engineering Applications (Ministry of Education), Shanghai Key Laboratory of PMMP,  East China Normal University, Shanghai 200241, China}
   \email{y.j0@qq.com}
\email{gdzhou@math.ecnu.edu.cn}

\date{\today}

\begin{abstract} Generalising a previous work of Jiang and Sheng, a cohomology theory for  differential Lie algebras of arbitrary weight is introduced.  The  underlying $L_\infty[1]$-structure on the cochain complex is also determined via  a generalised version of higher derived brackets. The equivalence between   $L_\infty[1]$-structures  for absolute and relative     differential Lie algebras  are established.   Formal deformations and abelian extensions are interpreted  by using lower degree cohomology groups. Also we introduce the homotopy differential Lie algebras. In a forthcoming paper, we will show that the operad of homotopy  (relative) differential Lie algebras is the minimal model of the operad of (relative)  differential Lie algebras.

\end{abstract}

\subjclass[2010]{
16E40   
16S80   
12H05   
12H10   
16W25   
16S70}  

\keywords{abelian extension, cohomology,   deformation, derived bracket, differential Lie algebra,  $L_\infty[1]$-algebra}

\maketitle

\tableofcontents

\allowdisplaybreaks

\section*{Introduction}

This paper studies the cohomology theory, abelian extensions, formal deformations, and $L_\infty$-structures  for differential Lie algebras of arbitrary  weight.

\subsection{Deformation theory and minimal models}\

An important method to study  a mathematical object is to   investigate properties of this   object under small deformations.
Algebraic deformation theory originated from the   deformation theory of complex structures due to   Kodaira and Spencer and evolved from    ideas of Grothendieck, Illusie, Gerstenhaber, Nijenhuis, Richardson, Deligne, Schlessinger, Stasheff, Goldman, Millson, etc.
A fundamental idea of deformation theory is that  the deformation theory
of any given mathematical object can be governed  by a certain differential graded Lie algebra or more
generally an $L_\infty$-algebra associated to this mathematical object  (whose underlying complex is called the deformation complex).
This philosophy has been realised as  a theorem in characteristic zero by Lurie~\cite{Lur} and Pridham~\cite{Pri10}. It is an important problem   to determine  explicitly this differential graded Lie algebra or $L_\infty$-algebra governing deformation theory of this mathematical object.

Another important problem about  algebraic structures is to study their homotopy versions, just like $A_\infty$-algebras for  usual associative algebras.
Expected  result would be providing a minimal model of the operad governing this algebraic structure.
When this operad  is Koszul,  the  Koszul duality for operads \cite{GK94, MSS02, LV12} gives a perfect solution. More precisely,
 the cobar construction of the Koszul dual cooperad  of the operad in question  is the  minimal model of this operad. However, when the operad   is NOT Koszul, essential difficulties arise and   few examples of minimal models   have been  worked out.

These two problems, say, describing controlling $L_\infty$-algebras  and constructing minimal models or more generally cofibrant resolutions,  are closed related. In fact, given a cofibrant resolution, in particular,  a minimal model,  of the operad in question, one can deduce from the cofibrant resolution  the deformation complex as well as   its $L_\infty$-structure as explained by Kontsevich and Soibelman \cite{KS99}.

However, in practice, a minimal model or a small cofibrant resolution is not known a priori.  Wang and Zhou \cite{WZ21, WZ22} recently found a method to solve these two problems for a large class of non-Koszul operads.  The method is in fact the original method of Gerstenhaber \cite{Ger63, Ger64}. We start by studying   formal deformations of the algebraic structure in question and construct the deformation complex as well as $L_\infty$-structure on it from deformation equations. Then we consider the homotopy version which corresponds to Maurer-Cartan elements in the $L_\infty$-structure when appropriate spaces are graded. At last we show the differential graded operad governing the homotopy version is the minimal model of the operad of this algebraic structure, and we also found the Koszul dual which,  in this case, is usually a homotopy cooperad.

We aim to deal with differential Lie algebras of arbitrary weight in this paper and its sequel by using the method of \cite{WZ21, WZ22} as well as derived bracket technique.

\subsection{Differential Lie algebras, old and new}\

As is well known,  differential operators appeared in  Calculus several centuries ago.  The study of differential algebras themselves, as the algebraic study of differential equations,  only  began in the 1930s at the hands of  Ritt \cite{Rit34,Rit50}.   From then on, by  the work of many mathematicians in the following decades, the subject has been fully developed into a vast area in mathematics.

 By definition a \textbf{differential algebra} is an associative commutative algebra $A$ endowed with a linear operator $\dd$  subject to the Leibniz rule:
 $$\dd(xy)=\dd(x)y+x\dd(y), x, y\in A.$$
 The discretisation of differential operators  are
  \textbf{difference operators}.  The action of   the difference operator $\Delta$ on a real function $f$  defined by:
$$(\Delta f)(x)=f(x+\Delta x)-f(x),$$
which satisfies
  $$\Delta(fg)=\Delta(f)g+f\Delta(g)+\Delta(f)\Delta(g).$$

In order to unify the study of differential operator and difference operator, Guo  and Keigher introduced weighted differential operators \cite{GK08}.
\begin{defn} (\cite{GK08})
	Let $\lambda\in \bk$ be a  fixed element. A  {\bf differential associative   algebra of weight $\lambda$}   is an associative  algebra $A$ together with a  linear operator $\dd  : A\rar A$ such that
	$$\dd(xy)=\dd(x)y+x\dd(y)+\lambda \dd(x)\dd(y), x, y\in A.$$
\end{defn}
So an usual differential operator is a differential operator of weight $0$ and a difference operator is a differential operator of weight $1$.

One of the motivations of   Guo and his collaborators to introduce the concept of weighted differential operators is to study   Rota's  program. Rota \cite{Rota}  proposed to classify  ``interesting" operators on an algebra. In order to promote this program,    Guo et al.  introduced the Gr\"{o}bner-Shirshov basis  theory for    algebras endowed with linear operators \cite{Guo09a, BCQ10, GSZ13, GaoGuo17}, and focused on two types of operator algebras: Rota-Baxter type algebras \cite {ZGGS21} and differential type algebras \cite{GSZ13}.  Weighted differential algebras fall into the second type.

Inspired by the above research, Guo and  Keigher \cite{GK08} introduced  differential Lie algebras of arbitrary weight.
\begin{defn} (\cite{GK08})
	Let $\lambda\in \bk$ be a  fixed element. A  {\bf differential Lie algebra of weight $\lambda$}   is a Lie  algebra $(\frakg,[~,~])$ together with a  linear operator $\dd_\frakg : \frakg\rar \frakg$ such that
	\begin{equation*}
		\dd_\frakg([x, y])=[\dd_\frakg(x), y]+[x,  \dd_\frakg(y)]+ \lambda[\dd_\frakg(x), \dd_\frakg(y)], \quad\forall x,y\in \frakg.
	\end{equation*}
\end{defn}

A related notion is that of relative difference Lie algebras.
\begin{defn}[{\cite{CC22, JS23}}]
A \textbf{LieAct triple} is the triple $(\frakg,\frakh,\rho)$, where $(\frakg,[~,~]_\frakg)$ and $(\frakh,[~,~]_\frakh)$ are Lie algebras and $\rho:\frakg\lon \mathrm{Der}(\frakh)$ is a homomorphism of Lie algebras, where $\mathrm{Der}(\frakh)$ is the space of derivations on $\frakh$.

Let  $(\frakg,\frakh,\rho)$ be a  LieAct triple. A linear map $D:\frakg\lon\frakh$ is called a \textbf{relative differential operator of weight} $\lambda$ if the following equality holds:
\begin{equation*}
	D([x, y]_\frakg)=\rho(x)(D(y))-\rho(y)(D(x))+\lambda [D(x), D(y)]_\frakh, \quad\forall x,y\in \frakg.
	\end{equation*}
\end{defn}

There has been some papers working  on   differential Lie algebras with nonzero weight before, but they mainly focus on relative differential Lie algebras of weight $1$, also called relative difference algebras. When studying non-Abelian extensions of Lie algebras, Lue~ \cite{Lue66}~introduced the concept of cross homomorphisms which are  just   relative difference operators  of weight $1$.
Caseiro~and~Costa~ \cite{CC22}~studied the framework for the existence of cross homomorphisms: LieAct triple $( \frakg,  \frakh,  \rho) $, and gave a differential graded Lie algebra where the~Maurer-Cartan~elements correspond one-to-one to Lie action triples.
Pei,  Sheng, Tang, and  Zhao~\cite{PSTZ21} defined the cohomology theory of cross homomorphisms on Lie algebras, and used the derived bracket technique to find the deformation of a differential graded Lie algebra where the~Maurer-Cartan~elements correspond one-to-one to the cross homomorphisms.
Jiang and  Sheng constructed an $L_\infty$-algebra for relative difference Lie algebra using the technique of derived brackets in \cite {JS23} and they also introduced cohomology theory for relative difference Lie algebras and absolute   difference Lie algebras.

This paper is different from papers of  Sheng and his collaborators, firstly generalising from  weight $1$ to arbitrary weight. Moreover, our ultimate goal  is to develop systematically the homotopy theory for  both  absolute and relative  differential Lie algebras of arbitrary weight  from an operadic viewpoint,  including  cohomology theory (or  deformation complexes),  $L_\infty[1]$-structures, homotopy versions,   minimal models,   and Koszul dual homotopy cooperads.
This task will be completed in two papers. This paper is the elementary part of this project, which contains only cohomology theory (or  deformation complexes),  $L_\infty[1]$-structures and homotopy versions.
In a forthcoming paper, we shall prove the operad of homotopy  absolute (resp. relative)      differential Lie algebras is the minimal model  of the operad of   absolute (resp. relative)     differential Lie algebras, thus justifying the homotopy version we found is the right one.

Apart from the cohomology theory and $L_\infty[1]$-structures, in  this paper   we introduce a generalised version of derived bracket technique which replace Lie subalgebras by injective maps of Lie algebras. Moreover, this generalised version  permits us to establish the equivalence between $L_\infty[1]$-structures for relative and absolute differential Lie algebras.  This equivalence is believed by experts and firstly presented explicitly in this paper for differential Lie algebras.  In the forthcoming paper, we will establish the equivalence between minimal models for relative and absolute differential Lie algebras by using a coloring and decoloring procedure for (coloured) operads.

\subsection{Layout of the paper}\


The paper is organized as follows.

The first section is of preliminary nature, which contains results which are more or less known. After fixing some notations in Subsection~\ref{Subsect: notations},  we recall basic notions and facts about  differential Lie  algebra and their representations in Subsection~\ref{Subsect:  representations}. In particular, a key descending property for differential representations Lemma~\ref{lem:bmd} is presented.
A differential Lie algebra is the combination of the underlying Lie algebra and the differential operator. In this light we build the cohomology theory of a differential Lie algebra by combining its components from the Lie algebra and from the differential operator.
In Subsection~\ref{subsec:cohomologydo}, we establish the cohomology theory for differential operators of arbitrary  weight, which is quite different from the one for the underlying algebra unless the weight is zero.
In Subsection~\ref{subsec:cohomologydasub}, we combine the Chevalley-Eilenberg cohomology for Lie algebras and the just established cohomology for differential operators of arbitrary  weight to define the cohomology of differential Lie algebras of arbitrary  weight, with the cochain maps again posing extra challenges when the weight is not zero.  The proofs of this section will be given in Subsection~\ref{Subsect: Cohomology of differential Lie algebras from  Linfinit structure}, once the $L_\infty[1]$-structure has been established in Section~\ref{sec:L_infty}.

As applications and further justification of our cohomology theory for differential Lie algebras, in Section \ref{sec:ext}, we apply the theory to study abelian extensions of differential Lie algebras of arbitrary  weight, and show that abelian extensions are classified by the second cohomology group of the differential Lie algebras.

Further, in Section \ref{sec:def}, we apply the above cohomology theory to study formal deformations of differential Lie algebras of arbitrary  weight. In particular, we show that if  the second cohomology group of a differential Lie algebra with coefficients in the regular representation is trivial, then this differential Lie algebra is rigid.

To deal with the weight case and get $L_\infty[1]$-structure in Section \ref{sec:L_infty}, after a reminder on 
 $L_\infty[1]$-structures in Subsection~\ref{Subsect: Linfinity algebras},   a generalised version of the derived bracket technique is introduced in Subsection~\ref{Subsect: derived bracket technique},  and  the weight $\lambda$  is  also incorporated into the statements,  which permits us to deal with (absolute) differential Lie algebras of   arbitrary weight as well as relative differential Lie algebras of   arbitrary weight.
With this generalised version at hand, we  build the $L_\infty[1]$-structure of    relative differential Lie algebras in Subsection~\ref{Subsect: Linifnity structure for relative  differential Lie algebras},
  generalising the result of Jiang and Sheng \cite{JS23} from weight $1$   to arbitrary weight.
  In Subsection~\ref{Subsect:  Linfinty for differential Lie algebras},  we  build the $L_\infty[1]$-structure of   (absolute) differential Lie algebras.  The comparison with the $L_\infty[1]$-structures   for relative  and absolute differential Lie algebras is provided in Subsection~\ref{Subsect: Relative vs absolute}.

At last in Section~\ref{Sect: Applications}, based on the  $L_\infty[1]$-structure of   (absolute) differential Lie algebras found in Section~\ref{sec:L_infty}, we verify that the cohomoloy theory of Section~\ref{sec:cohomologyda} is the right one in Subsection~\ref{Subsect: Cohomology of differential Lie algebras from  Linfinit structure}, by using a trivial extension construction in order to deal with coefficients.  In Subsection~\ref{Subsect: Homotopy  differential Lie algebras}, the structure of homotopy differential Lie algebra with weight is introduced as Maurer-Cartan elements in the    $L_\infty[1]$-structure of   (absolute) differential Lie algebras when the spaces involved are graded.

\bigskip
\section{Differential Lie  algebras of arbitrary weight and their cohomology theory} \label{sec:cohomologyda}

This section recalls some background on differential Lie algebras and  their representations and introduce their cohomology theory generalsing the theory introduced by Jiang and Sheng \cite{JS23} from weight $1$ to arbitrary weight.

\subsection{Notations}\ \label{Subsect: notations}

Throughout this paper, let $\bfk$ be a field of characteristic $0$.  Except specially stated,  vector spaces are  $\bfk$-vector spaces and  all    tensor products are taken over $\bfk$.  We use cohomological grading.
%
%
%
 We will employ Koszul sign rule to determine signs,  that is, when  exchanging the positions of two graded objects in an expression, we have  to multiply the expression by a power of $-1$ whose  exponent is  the product of their degrees.

The  suspension  of a graded space $V=\oplus_{n\in \ZZ} V^n$ is the graded space  $sV $ with $(sV)^n=V^{n+1}$ for any $n\in \ZZ$.
Write $sv\in (sV)^n$ for $v\in V^{n+1}$.

 Let $V=\oplus_{n\in \mathbb{Z}} V^n$ be a graded vector space. Recall that   the graded symmetric algebra $S(V)$ of $V$ is defined to be the quotient of the tensor algebra $T(V)$ by   the two-sided ideal $I$   generated by
$x\ot y -(-1)^{|x||y|}y\ot x$ for all homogeneous elements $x, y\in V$. For $x_1\ot\cdots\ot x_n\in V^{\ot n}\subseteq T(V)$, write $ x_1\odot x_2\odot\dots\odot x_n$ its image in $S(V)$.
For homogeneous elements $x_1,\dots,x_n \in V$ and $\sigma\in S_n$ which is  the symmetric group in $n$ variables, the Koszul sign $\varepsilon(\sigma):=\varepsilon(\sigma;  x_1,\dots, x_n)$ is defined by
$$ x_1\odot x_2\odot\dots\odot x_n=\varepsilon(\sigma)x_{\sigma(1)}\odot x_{\sigma(2)}\odot\dots\odot x_{\sigma(n)}\in S(V).$$
Denote by $S^n(V)$ the image of $V^{\otimes n}$ in $S(V)$.

Let $V=\oplus_{n\in \mathbb{Z}} V^n$ be a graded vector space. Recall that   the graded exterior  algebra $\bigwedge(V)$ of $V$ is defined to be the quotient of the tensor algebra $T(V)$ by   the two-sided ideal generated by
$x\ot y +(-1)^{|x||y|}y\ot x$ for all homogeneous elements $x, y\in V$. For $x_1\ot\cdots\ot x_n\in V^{\ot n}\subseteq T(V)$, write $ x_1\wedge x_2\wedge\cdots\wedge x_n$ its image in $\bigwedge(V)$.
For homogeneous elements $x_1,\dots,x_n \in V$ and $\sigma\in S_n$, the   sign $\chi(\sigma):=\chi(\sigma;  x_1,\dots, x_n)$ is defined by
$$  x_1\wedge x_2\wedge\cdots\wedge x_n=\chi(\sigma)x_{\sigma(1)}\wedge x_{\sigma(2)}\wedge\cdots\wedge x_{\sigma(n)}\in \bigwedge(V).$$
Denote by $\bigwedge^n(V)$ the image of $V^{\otimes n}$ in $\bigwedge(V)$.
Obviously $\chi(\sigma)=\varepsilon(\sigma)\sgn(\sigma)$, where $\sgn(\sigma)$ is the signature of $\sigma$.

Recall that for a graded vector space $V$,  $S^n(sV)$ is isomorphic to  $s^n\wedge^nV$ via the map
\begin{equation}\label{Eq: exterior vs symmetric}sv_1\odot sv_2\odot\dots\odot sv_n\mapsto (-1)^{(n-1)|v_1|+(n-2)|v_2|+\cdots+|v_{n-1}|} s^n( v_1\wedge v_2\wedge \cdots \wedge v_n),\end{equation}
for $v_1, \dots, v_n\in V$ homogeneous elements in $V$.

 Let  $n\geq 1$.
 For $0\leq i_1, \dots, i_r\leq n$ with $i_1+\cdots+i_r=n$,  $\Sh(i_1, i_2,\dots,i_r)$ is the   set of $(i_1,\dots, i_r)$-shuffles, i.e., those permutation $\sigma\in S_n$ such that
 		 $$\sigma(1)<\sigma(2)<\dots<\sigma(i_1),  \ \sigma(i_1+1)< \dots<\sigma(i_1+i_2),\ \dots,\
 		\sigma(i_1+\cdots+i_{r-1}+1)< \cdots<\sigma(n).$$
 Let $\mathrm{PSh}(i_1, i_2,\dots,i_r)$   be the subset of $\Sh(i_1, i_2,\dots,i_r)$ whose elements $(i_1,\dots, i_r)$-shuffles satisfy:$$\sigma(1)<\sigma(i_1+1)<\sigma(i_1+i_2+1)<\dots<\sigma(i_1+\cdot +i_{r-1}+1).$$

 \subsection{Differential Lie algebras with weight and their representations}\
 \label{Subsect:  representations}

 In this subsection, we recall basic notions about differential Lie algebras with weight and their representations.

\begin{defn} (\cite{GK08})
	Let $\lambda\in \bk$ be a  fixed element. A  {\bf differential Lie algebra of weight $\lambda$}   is a Lie  algebra $(\frakg,~[~,~])$ together with a  linear operator $\dd_\frakg : \frakg\rar \frakg$ such that
	\begin{equation}\label{Eq: diff}
		\dd_\frakg([x, y])=[\dd_\frakg(x), y]+[x,  \dd_\frakg(y)]+\lambda [\dd_\frakg(x), \dd_\frakg(y)], \quad\forall x,y\in \frakg.
	\end{equation}
	Such an operator $\dd_\frakg$ is called a {\bf differential operator of weight $\lambda$} or a  {\bf derivation of weight $\lambda$}.
	
	Given two differential Lie algebras $(\frakg, \dd_\frakg),\,(\frakh,\dd_\frakh)$ of the same  weight $\lambda$, a {\bf homomorphism of differential Lie algebras} from $(\frakg,\dd_\frakg)$ to $(\frakh,\dd_\frakh)$ is a Lie algebra homomorphism $\varphi:\frakg\lon \frakh$ such that $\varphi\circ \dd_\frakg=\dd_\frakh\circ\varphi$.
	We denote by $\Diffl$ the category of  differential Lie algebras of weight $\lambda$.
\end{defn}

A differential operator of weight $1$ is also called a crossed homomorphism and a differential Lie algebra of weight $1$ is also called a difference Lie algebra \cite{PSTZ21, JS23}.


Recall that a representation of a Lie algebra $\frakg$ is a pair $(V,\rho)$, where $V$ is a vector space and $\rho: \frakg\to \mathfrak{gl}(V),~ x\mapsto (v\mapsto \rho(x)v)$ is a homomorphism  of Lie algebras for all $x,y\in \frakg$ and $v\in V$.

\begin{defn}
	Let $(\frakg, \dd_\frakg)$ be a differential Lie algebra.
	\begin{itemize}
		\item[{\rm (i)}] A {\bf representation} over the differential Lie algebra $(\frakg, \dd_\frakg)$ or a \textbf{differential representation} is a triple $(V,\rho, \dd_V)$, where $\dd_V\in \mathfrak{gl}(V)$, and  $(V,\rho)$ is a representation over the Lie algebra $\frakg$, such that for all $x,y\in \frakg, v\in V,$
		the following equality holds:
		\begin{eqnarray*}
			\dd_V(\rho(x)v)&=&\rho(\dd_\frakg(x))v+\rho(x)\dd_V(v)+\lambda \rho(\dd_\frakg(x))\dd_V(v).
		\end{eqnarray*}
		
		\item[{\rm (ii)}] Given two representations $(U,\rho^U, \dd_U),\,(V,\rho^V, \dd_V)$ over  $(\frakg, \dd_\frakg)$, a  linear map $f:U\lon V$ is called a {\bf homomorphism} of representations, if$$f\circ\rho^U(x)=\rho^V(x)\circ f,\quad \forall x\in \frakg, \quad \mathrm{and}\quad f\circ \dd_U=\dd_V\circ f.$$
	\end{itemize}
\end{defn}


\begin{exam}For a differential Lie algebra $(\frakg, \dd_\frakg)$,    the usual adjoint representation
	$$\ad:\frakg\to \mathfrak{gl}(\frakg),\,x\mapsto(y\mapsto [x, y])$$
	over the Lie algebra $\frakg$ is also  a representation of this differential Lie algebra.
	It is also  called the {\bf adjoint representation} over the differential Lie algebra $(\frakg, \dd_\frakg)$, denoted by $\frakg_{\ad}$.
\end{exam}

\begin{exam}
	Let $\frakg$ be a  Lie algebra and  $V$ be a  representation of it. Then  $(\frakg, \mathrm{Id}_\frakg)$  is a   differential Lie algebra  of weight $-1$ and  the pair $(V, \mathrm{Id}_V)$ gives  a  differential representation.
\end{exam}

\begin{exam}
	Let $(\frakg, \dd_\frakg)$ be a differential Lie algebra together with   a differential representation   $(V, \dd_V)$. Then  for arbitrary  nonzero scalar $\kappa\in \bk$,  $(\frakg, \kappa  \dd_\frakg)$ is a differential Lie algebra  of weight $\frac{\lambda}{\kappa}$ and    the pair $(V, \kappa\  \dd_V)$  is a differential representation.
	
\end{exam}

It is straightforward to obtain the following result:
\begin{prop}\label{Prop: trivial extensions}
	Let $(V,\rho, \dd_V)$ be a representation of the  differential Lie algebra $(\frakg, \dd_\frakg)$. Then $(\frakg\oplus V, \dd_\frakg+ \dd_V)$ is a differential Lie algebra, where the Lie algebra structure on $\frakg\oplus V$ is given by
	$$
	\{x+u,y+v\}:=[x, y]+\rho(x)v-\rho(y)u,\quad \forall x,y\in \frakg,~u,v\in V.
	$$
This new differential Lie algebra is denoted by $\frakg\ltimes V$, called the \textbf{trivial extension} of $\frakg$ by $V$.
\end{prop}

We need the following observation which is a weighted version of \cite[Proposition 3.1]{Lue66}:
\begin{lemma}\label{lem:bmd}
Let $(V,\rho, \dd_V)$ be a representation of the  differential Lie algebra $(\frakg, \dd_\frakg)$. There exists	a new representation $(V,\rho_\lambda,\dd_V)$  over $(\frakg, \dd_\frakg)$, where  $\rho_\lambda$ is given by$\colon$
	\[\rho_\lambda(x) v=\rho(x+\lambda \dd_\frakg(x))v,\quad \forall x \in \frakg, v\in V.\]
	Denote the new presentation  by $V_\lambda:=(V, \rho_\lambda, \dd_V)$.
\end{lemma}



\subsection{Cohomology of differential operators}\label{subsec:cohomologydo}\

In this subsection, we define a cohomology theory  of differential operators.

Recall that
the  Chevalley-Eilenberg cochain complex of a Lie algebra $\frakg$ with coefficients in a representation $(V,\rho)$  is the cochain complex  $(\rmC^*_\Alg(\frakg, V),\partial_{\Alg}^*),$
where  for $n\geq 0$, $\rmC^n_\Alg(\frakg,   V)=\Hom(\wedge^n\frakg,  V)$ (in particular, $\rmC^0_\Alg(\frakg, V)=V $) and  the coboundary operator $$\partial_{\Alg}^n: \rmC^n_\Alg(\frakg, V)\longrightarrow  \rmC^{n+1}_\Alg(\frakg, V), n\geq 0$$ is given by
\[\begin{split}
	\partial_{\Alg}^n (f)(x_1,\dots,x_{n+1})&=\sum_{i=1}^{n+1}(-1)^{i+n}\rho(x_i) f(x_1,\dots,\hat{x}_i, \dots, x_{n+1})\\
	&+\sum_{1\leq i<j\leq n+1} (-1)^{i+j+n+1}f([x_i, x_j], x_1,\dots,\hat{x}_i, \dots,\hat{x}_j, \dots,x_{n+1}),
\end{split}\]
for all $f\in \rmC^n_\Alg(\frakg, V ),~x_1,\dots, x_{n+1}\in \frakg$, where $\hat{x_i}$ means deleting this element.  The corresponding  Chevalley-Eilenberg cohomology is denoted by  $\rmH^*_\Alg(\frakg, V )$.

When $V $ is the  adjoint representation $\frakg_{\ad}$, we write $\rmH^n_\Alg(\frakg):=\rmH^n_\Alg(\frakg, V ), n\geq 0$.

\begin{defn}  Let $(\frakg,\dd_\frakg)$ be a differential Lie algebra of weight $\lambda$  and  $(V, \dd_V)$ be a differential  representation over $(\frakg, \dd_\frakg)$.
	The \textbf{cochain complex   of the differential operator $\dd_\frakg$ with coefficients in  the differential representation $(V,\dd_V)$} is defined to be the Chevalley-Eilenberg cochain complex of the  Lie algebra $\frakg$ with coefficients in the new differential representation $V_\lambda$, denoted by $(\rmC^*_\diffo(\frakg,V), \partial_\DO^*):=(\rmC^*_\Lie(\frakg,V_\lambda), \partial_\Lie^*)$.
	The corresponding cohomology, denoted by $\rmH^*_\diffo(\frakg,V)$, is called the {\bf cohomology of the differential operator $\dd_\frakg$ with coefficients in the representation} $(V_\lambda,\dd_V)$.
\end{defn}
More precisely, for $n\geq 0$, $\rmC^n_\diffo(\frakg,V)=\Hom(\wedge^n\frakg,  V)$ and the coboundary operator

$\partial_\DO^n: \rmC^n_\diffo(\frakg,V)\to \rmC^{n+1}_\diffo(\frakg,V)$ is given by
$$\begin{aligned}
	\partial_\DO^n  (f)(x_1,\dots,x_{n+1})&=\sum_{i=1}^{n+1}(-1)^{i+n}\rho_\lambda(x_i) f(x_1,\dots,\hat{x}_i, \dots, x_{n+1})\\
	&\quad+\sum_{1\leq i<j\leq n+1}(-1)^{i+j+n+1}f([x_i, x_j], x_1,\dots,\hat{x}_i, \dots,\hat{x}_j, \dots,x_{n+1})\\
	&=\sum_{i=1}^{n+1}(-1)^{i+n}\rho(x_i+\lambda \dd_\frakg(x_i)) f(x_1,\dots,\hat{x}_i, \dots, x_{n+1})\\
	&\quad+\sum_{1\leq i<j\leq n+1}(-1)^{i+j+n+1}f([x_i, x_j], x_1,\dots,\hat{x}_i, \dots,\hat{x}_j, \dots,x_{n+1}),
\end{aligned}$$
for all $f\in \rmC^n_\diffo(\frakg, V),~x_1,\dots,x_{n+1}\in \frakg$.

\bigskip

\subsection{Cohomology of differential Lie algebras} \label{subsec:cohomologydasub}\

We now combine the classical Chevalley-Eilenberg  cohomology of Lie algebras and the newly defined cohomology of differential operators of weight $\lambda$ to  define the cohomology of the differential Lie algebra $(\frakg, \dd_\frakg)$ with coefficients in the representation $(V,\dd_V)$.  In fact, we will define the cochain complex of a differential Lie algebras as, up to shift and signs, the mapping cone of a cochain map from the Chevalley-Eilenberg  cochain complex of the Lie algebra to the cochain complex of the differential operators of weight $\lambda$.

Notice that \cite{JS23} has introduced the cohomology theory of a difference Lie algebra  with coefficients in a representation.  We in fact generalise their construction from weight $1$ to arbitrary weight.

Introduce   the linear maps
$$\delta^n: \rmC^n_\alg(\frakg,V)\rar \rmC^n_\DO(\frakg,  V)$$
by
\begin{align*}
	\delta^n f(x_1,\dots,x_n):=\sum_{k=1}^n\lambda^{k-1}\sum_{1\leq i_1<\cdots<i_k\leq n}f(x_1,\dots,\dd_\frakg(x_{i_1}),\dots,\dd_\frakg(x_{i_k}),\dots,x_n)-\dd_V f(x_1,\dots,x_n),
\end{align*}
for  $f\in \rmC^n_\alg(\frakg,V)$, $  n\geq 1$ and $$\delta^0 v= {-} \dd_V {(v)},\quad \forall v\in \rmC^0_\alg(\frakg,V)=V.$$

\begin{prop}\mlabel{prop:delta}
	The linear map $\delta$ is a cochain map from the cochain complex $(\rmC^*_\alg(\frakg,V),\partial_\alg^*)$ to $(\rmC^*_\DO(\frakg, V),\partial_\DO^*)$.
\end{prop}

Although a direct proof is possible, we shall deduce the above result   from the  $L_\infty[1]$-structure; see Proposition~\ref{dlcohomology}.


\begin{remark}
	Note that $\rmC^n_\DO(\frakg,V)$ equals to $\rmC^n_\alg(\frakg,V)$ as linear spaces but they are not equal as cochain complexes unless $\lambda=0$. When $\lambda$ is not zero, a new representation structure is needed to define $\partial_\DO^*$ which eventually leads to the rather long and technical argument in order to establish the cochain map in  Proposition~\mref{prop:delta}.
	\mlabel{rk:nonzero}
\end{remark}
Now we can define the cochain complex of a differential Lie algebras $(\frakg,\dd_\frakg)$ with coefficients in the  differential representation $(V,\dd_V)$.

\begin{defn}\label{def: cochain complex for differential Lie algebras}
	Define the \textbf{cochain complex $(\rmC_{\Diffl}^*(\frakg, V),\partial^*_{\Diffl})$ of the differential Lie algebra $(\frakg, \dd_\frakg)$ with coefficients in the differential representation} $(V,\dd_V)$  to be the negative shift of
	the mapping cone of the cochain map $\delta^{*}: (\rmC^*_\alg(\frakg,V),{\partial}_\alg^*) \rar (\rmC^*_\DO(\frakg,V), {\partial}_\DO^*)$. More precisely, the space of $n$-cochains is given by
	\begin{equation*}\label{eq:dac}
		\rmC_{\Diffl}^n(\frakg,V):=
		\begin{cases}
			\rmC^n_\alg(\frakg,V)\oplus \rmC^{n-1}_\DO(\frakg,V),&n\geq1,\\
			\rmC^0_\alg(\frakg,V)=V,&n=0.
		\end{cases}
	\end{equation*}
	and   the differential   $\partial_{\Diffl}^n: \rmC_{\Diffl}^n(g,V)\rar \rmC_{\Diffl}^{n+1}(g,V)$ by
	\begin{eqnarray*}\label{eq:pda}
		\partial_{\Diffl}^n(f,g)&:=&(\partial_\alg^n f,  -\partial_\DO^{n-1} g-\delta^n f),\quad \forall f\in \rmC^n_\alg(\frakg,V),\,g\in \rmC^{n-1}_\DO(\frakg,V),\quad n\geq 1,\\
		\label{eq:pda2}\partial_{\Diffl}^0 {v}&:=&(\partial_\alg^0 {v}, \delta^0 v),\quad\forall {v}\in \rmC^0_\alg(\frakg,V)=V.
	\end{eqnarray*}	
\end{defn}

%

\begin{defn}\label{def:dlcohomology}
	The cohomology of the cochain complex $(\rmC_{\Diffl}^*(\frakg, V),\partial^*_{\Diffl})$, denoted by $\rmH_{\Diffl}^*(\frakg, V)$, is called the {\bf cohomology of the differential Lie algebra} $(\frakg, \dd_\frakg)$ with coefficients in the differential representation $(V,\dd_V)$.
\end{defn}

In the next two sections, we will use the following remarks.
\begin{remark}\label{ex:cocycle}
	We compute 0-cocycles, 1-cocycles and 2-cocycles of the cochain complex
	$\rmC_{\Diffl}^*(\frakg,V)$.
	
	It is obvious that for all $v\in V$, $\partial^0_{\Diffl} v=0$ if and only if
	$$\partial^0_\alg  v=0,\quad \dd_V(v)=0.$$
	
	For all $(f,v)\in\Hom (\frakg,V)\oplus V$, $\partial^1_{\Diffl} (f,v)=0$ if and only if $\partial^1_\alg  f=0$ and
	$$
	\rho_\lambda(x) v=f(\dd_\frakg(x))-\dd_V(f(x)),\quad \forall x\in \frakg.
	$$
	
	For all $(f,g)\in\Hom (\wedge^2\frakg,V)\oplus\Hom (\frakg,V) $,  $\partial^2_{\Diffl}(f,g)=0$ if and only if $\partial^2_\alg f=0,$ and
	\begin{align}\label{2-cocycle}
		\rho_\lambda(x) g(y)- \rho_\lambda(y) g(x)-g([x, y])=  -\lambda f(\dd_\frakg(x),\dd_\frakg(y))-f(\dd_\frakg(x),y)-f(x,\dd_\frakg(y))+\dd_V(f(x,y)),\
	\end{align}
	for all $x,y\in \frakg.$
\end{remark}
\begin{remark}
	we shall need a subcomplex of the cochain complex $\rmC_{\Diffl}^*(\frakg, V)$. Let
	\begin{equation*}\label{eq:dacsub}
		\tilde{\rmC}_{\Diffl}^n(\frakg,V):=
		\begin{cases}
			\rmC^n_\alg(\frakg,V)\oplus \rmC^{n-1}_\diffo(\frakg,V),&n\geq2,\\
			\rmC^1_\alg(\frakg,V),&n=1,\\
			0,&n=0.
		\end{cases}
	\end{equation*}
	Then it is obvious that $(\tilde{\rmC}_{\Diffl}^*(\frakg,V)=\oplus_{n=0}^\infty\tilde{\rmC}_{\Diffl}^n(\frakg,V),\partial_{\Diffl}^* )$  is a subcomplex of the cochain complex  $(\rmC_{\Diffl}^*(\frakg, V),\partial_{\Diffl}^*)$. We  denote its cohomology by $\tilde{\rmH}_{\Diffl}^*(\frakg,V)$. Obviously, $\tilde{\rmH}_{\Diffl}^n(\frakg,V)= {\rmH}_{\Diffl}^n(\frakg,V)$ for $n>2$.
	
\end{remark}

\section{Abelian extensions of differential Lie  algebras} \label{sec:ext}
In this section, we study abelian extensions of differential Lie  algebras of weight $\lambda$ and show that they are classified by the second cohomology, as one would expect of a good cohomology theory.

\subsection{Abelian extensions}\

Notice that   a vector space    $V$  together   with a linear endomorphism $\dd_V$   can be considered as a \textbf{trivial differential Lie  algebra of weight} $\lambda$
endowed with the trivial Lie bracket $[u, v]=0$ for all $u,v \in  V$.
\begin{defn}
	An {\bf abelian extension} of differential Lie  algebras is a short exact sequence of homomorphisms of differential Lie algebras
	\[\begin{CD}
		0@>>> {V} @>i >> \hat{\frakg} @>p >> \frakg @>>>0\\
		@. @V \dd_{V} VV @V \dd_{\hat{\frakg}} VV @V \dd_\frakg VV @.\\
		0@>>> {V} @>i >> \hat{\frakg} @>p >> \frakg @>>>0
	\end{CD}\]
	where  $(V, \dd_V)$ is a trivial differential Lie algebra.
	We will call $(\hat{\frakg},\dd_{\hat{\frakg}})$ an \textbf{abelian extension} of $(\frakg,\dd_\frakg)$ by $(V,\dd_V)$.
\end{defn}

\begin{defn}
	Let $(\hat{\frakg}_1,\dd_{\hat{\frakg}_1})$ and $(\hat{\frakg}_2,\dd_{\hat{\frakg}_2})$ be two abelian extensions of $(\frakg,\dd_\frakg)$ by $(V,\dd_V)$. They are said to be {\bf isomorphic} if there exists an isomorphism of differential Lie algebras $\zeta:(\hat{\frakg}_1,\dd_{\hat{\frakg}_1})\rar (\hat{\frakg}_2,\dd_{\hat{\frakg}_2})$ such that the following commutative diagram holds:
	\[\begin{CD}
		0@>>> {(V,\dd_V)} @>i >> (\hat{\frakg}_1,\dd_{\hat{\frakg}_1}) @>p >> (\frakg,\dd_\frakg) @>>>0\\
		@. @| @V \zeta VV @| @.\\
		0@>>> {(V,\dd_V)} @>i >> (\hat{\frakg}_2,\dd_{\hat{\frakg}_2}) @>p >> (\frakg,\dd_\frakg) @>>>0.
	\end{CD}\]
\end{defn}

A {\bf section} of an abelian extension $(\hat{\frakg},\dd_{\hat{\frakg}})$ of $(\frakg,\dd_\frakg)$ by $(V,\dd_V)$ is a linear map $s:\frakg\rar \hat{\frakg}$ such that $p\circ s=\Id_\frakg$.

Now for an abelian extension $(\hat{\frakg},\dd_{\hat{\frakg}})$ of $(\frakg,\dd_\frakg)$ by $(V,\dd_V)$ with a section $s:\frakg\rar \hat{\frakg}$, then there exist a unique linear map $t:\hat{\frakg}\rar V$ such that:
$$t\circ i=\mathrm{Id}_V, t\circ s=0, i\circ t+s\circ p=\mathrm{Id}_{\hat{\frakg}},$$
then we can define $$
\rho(x)v:=t[s(x),i(v)]_{\hat{\frakg}}, \quad \forall x\in \frakg, v\in V.
$$
For convenience, we denote it by  $\rho(s(x))v,$ and observe that $[s(x),i(v)]_{\hat{\frakg}}\in i(V)$, and we identify $V$ with $i(V).$

Now we get a linear map $\rho: \frakg\to \mathfrak{gl}(V),~ x\mapsto (v\mapsto \rho(x)v) .$

\begin{prop}
	With the above notations, $(V,\rho, \dd_V)$ is a representation over the differential Lie algebra $(\frakg,\dd_\frakg)$.
\end{prop}
\begin{proof}
	For arbitrary  $x,y\in\frakg, v\in V$, since $s([x, y]_{\frakg})-[s(x), s(y)]_{\hat{\frakg}}\in V$ implies $s([x, y]_{\frakg})v=[s(x), s(y)]_{\hat{\frakg}}v$, we have
	\[\rho([x, y]_{\frakg})(v)=\rho(s([x, y]_{\frakg}))v=\rho([s(x), s(y)]_{\hat{\frakg}})v\xlongequal{\mbox{Jacobi identity}}\rho(x)\rho(y)(v)-\rho(y)\rho(x)(v).\]
	Hence, $\rho$ is a Lie algebra homomorphism.  Moreover, $\dd_{\hat{\frakg}}(s(x))-s(\dd_\frakg(x))\in V$ means that  $$\rho(\dd_{\hat{\frakg}}(s(x)))v=\rho(s(\dd_\frakg(x)))v.$$ Thus we have
	\begin{align*}
		\dd_V(\rho(x)v)&=\dd_V(\rho(s(x))v)=\dd_{\hat{\frakg}}(\rho(s(x))v)\\
		&=\rho(\dd_{\hat{\frakg}}(s(x)))v+\rho(s(x))\dd_{\hat{\frakg}}(v)+\lambda \rho(\dd_{\hat{\frakg}}(s(x)))\dd_{\hat{\frakg}}(v)\\
		&=\rho(s(\dd_\frakg(x)))v+\rho(s(x))\dd_V(v)+\lambda \rho(s(\dd_\frakg(x)))\dd_V(v)\\
		&=\rho(\dd_\frakg(x))v+\rho(x)\dd_V(v)+\lambda \rho(\dd_\frakg(x))\dd_V(v).
	\end{align*}
	Hence, $(V,\rho, \dd_V)$ is a representation over $(\frakg,\dd_\frakg)$.
\end{proof}

We  further  define linear maps $\psi:\frakg\otimes \frakg\rar V$ and $\chi:\frakg\rar V$ respectively by
\begin{align*}
	\psi(x,y)&=[s(x), s(y)]_{\hat{\frakg}}-s([x, y]_{\frakg}),\quad\forall x,y\in \frakg,\\
	\chi(x)&=\dd_{\hat{\frakg}}(s(x))-s(\dd_\frakg(x)),\quad\forall x\in \frakg.
\end{align*}
We transfer the differential Lie algebra structure on $\hat{\frakg}$ to $\frakg\oplus V$ by endowing $\frakg\oplus V$ with a multiplication $[\cdot, \cdot]_\psi$ and a differential operator $\dd_\chi$ defined by
\begin{align}
	\label{eq:mul}[x+u, y+v]_\psi&=[x, y]+\rho(x)v-\rho(y)u+\psi(x,y),\,\forall x,y\in \frakg,\,u,v\in V,\\
	\label{eq:dif}\dd_\chi(x+v)&=\dd_\frakg(x)+\chi(x)+\dd_V(v),\,\forall x\in \frakg,\,v\in V.
\end{align}

\begin{prop}\label{prop:2-cocycle}
	The triple $(\frakg\oplus V,[\cdot, \cdot]_\psi,\dd_\chi)$ is a differential Lie algebra   if and only if
	$(\psi,\chi)$ is a 2-cocycle  of the differential Lie algebra $(\frakg,\dd_\frakg)$ with the coefficient  in $(V,\dd_V)$.
\end{prop}
\begin{proof}
	If $(\frakg\oplus V,[\cdot, \cdot]_\psi,\dd_\chi)$ is a differential Lie algebra, then  $[\cdot, \cdot]_\psi$ implies
	\begin{equation}
		\label{eq:mc}\psi (x, [y,z])+\rho(x)\psi(y,z)+\psi (y, [z, x])\\
		+\rho(y)\psi(z,x)+\psi (z, [x,y])+\rho(z) \psi(x,y)=0,
	\end{equation}
for arbitrary  $x,y,z\in\frakg$. Since $\dd_\chi$ satisfies (\ref{Eq: diff}), we deduce that
	\begin{equation}
		\label{eq:dc}\chi([x, y])-\rho_\lambda(x)\chi(y)+\rho_\lambda(y)\chi(x) +\dd_V(\psi(x,y))-\psi(\dd_\frakg(x),y)-\psi(x,\dd_\frakg(y))-\lambda \psi(\dd_\frakg(x),\dd_\frakg(y))=0.
	\end{equation} Hence, $(\psi,\chi)$ is a  2-cocycle, see Eq.~(\ref{2-cocycle}).
	
	Conversely, if $(\psi,\chi)$ satisfies equalities~\eqref{eq:mc} and \eqref{eq:dc}, one can easily check that $(\frakg\oplus V,[\cdot, \cdot]_\psi,\dd_\chi)$ is a differential Lie algebra.
\end{proof}

\subsection{Classificaiton for Abelian extensions}\

Now we are ready to classify abelian extensions of a differential Lie algebra.

\begin{theorem}
	Let $V$ be a vector space and  $\dd_V\in\End_\bk(V)$.
	Then abelian extensions of a differential Lie algebra $(\frakg,\dd_\frakg)$ by $(V,\dd_V)$ are classified by the second cohomology group ${\rm\tilde{H}}_{\Diffl}^2(\frakg,V)$ of $(\frakg,\dd_\frakg)$ with coefficients in the representation $(V,\dd_V)$.
\end{theorem}
\begin{proof}
	Let $(\hat{\frakg},\dd_{\hat{\frakg}})$ be an abelian extension of $(\frakg,\dd_\frakg)$ by $(V,\dd_V)$. We choose a section $s:\frakg\rar \hat{\frakg}$ to obtain a 2-cocycle $(\psi,\chi)$ by Proposition~\ref{prop:2-cocycle}. We first show that the cohomological class of $(\psi,\chi)$ does not depend on the choice of sections. Indeed, let $s_1$ and $s_2$ be two distinct sections providing 2-cocycles $(\psi_1,\chi_1)$ and $(\psi_2,\chi_2)$ respectively. We define $\phi:\frakg\rar V$ by $\phi(x)=s_1(x)-s_2(x)$. Then
	\begin{align*}
		\psi_1(x,y)&=[s_1(x), s_1(y)]-s_1([x, y])\\
		&=[s_2(x)+\phi(x), s_2(y)+\phi(y)]-(s_2([x, y])+\phi([x, y]))\\
		&=([s_2(x), s_2(y)]-s_2([x, y]))+[s_2(x), \phi(y)]+[\phi(x), s_2(y)]-\phi([x, y])\\
		&=([s_2(x), s_2(y)]-s_2([x, y]))+[x, \phi(y)]+[\phi(x), y]-\phi([x, y])\\
		&=\psi_2(x,y)+\partial^1_\alg\phi(x,y),
	\end{align*}
	and
	\begin{align*}
		\chi_1(x)&=\dd_{\hat{\frakg}}(s_1(x))-s_1(\dd_\frakg(x))\\
		&=\dd_{\hat{\frakg}}(s_2(x)+\phi(x))-(s_2(\dd_\frakg(x))+\phi(\dd_\frakg(x)))\\
		&=(\dd_{\hat{\frakg}}(s_2(x))-s_2(\dd_\frakg(x)))+\dd_{\hat{\frakg}}(\phi(x))-\phi(\dd_\frakg(x))\\
		&=\chi_2(x)+\dd_V(\phi(x))-\phi(\dd_\frakg(x))\\
		&=\chi_2(x)-\delta^1\phi(x).
	\end{align*}
	That is, $(\psi_1,\chi_1)=(\psi_2,\chi_2)+\partial^1_{\Diffl}(\phi)$. Thus $(\psi_1,\chi_1)$ and $(\psi_2,\chi_2)$ are in the same cohomological class  {in $\tilde{\rm{H}}_{\Diffl}^2(\frakg,V)$}.
	
	Next we prove that isomorphic abelian extensions give rise to the same element in  {$\tilde{\rm{H}}_{\Diffl}^2(\frakg,V)$.} Assume that $(\hat{\frakg}_1,\dd_{\hat{\frakg}_1})$ and $(\hat{\frakg}_2,\dd_{\hat{\frakg}_2})$ are two isomorphic abelian extensions of $(\frakg,\dd_\frakg)$ by $(V,\dd_V)$ with the associated homomorphism $\zeta:(\hat{\frakg}_1,\dd_{\hat{\frakg}_1})\rar (\hat{\frakg}_2,\dd_{\hat{\frakg}_2})$. Let $s_1$ be a section of $(\hat{\frakg}_1,\dd_{\hat{\frakg}_1})$. As $p_2\circ\zeta=p_1$, we have
	\[p_2\circ(\zeta\circ s_1)=p_1\circ s_1=\Id_{\frakg}.\]
	Therefore, $\zeta\circ s_1$ is a section of $(\hat{\frakg}_2,\dd_{\hat{\frakg}_2})$. Denote $s_2:=\zeta\circ s_1$. Since $\zeta$ is a homomorphism of differential Lie algebras such that $\zeta|_V=\Id_V$, we have
	\begin{align*}
		\psi_2(x,y)&=[s_2(x), s_2(y)]-s_2([x, y])=[\zeta(s_1(x)), \zeta(s_1(y))]-\zeta(s_1([x, y]))\\
		&=\zeta([s_1(x), s_1(y)]-s_1([x, y]))=\zeta(\psi_1(x,y))\\
		&=\psi_1(x,y),
	\end{align*}
	and
	\begin{align*}
		\chi_2(x)&=\dd_{\hat{\frakg}_2}(s_2(x))-s_2(\dd_\frakg(x))=\dd_{\hat{\frakg}_2}(\zeta(s_1(x)))-\zeta(s_1(\dd_\frakg(x)))\\
		&=\zeta(\dd_{\hat{\frakg}_1}(s_1(x))-s_1(\dd_\frakg(x)))=\zeta(\chi_1(x))\\
		&=\chi_1(x).
	\end{align*}
	Consequently, all isomorphic abelian extensions give rise to the same element in {$\tilde{\rm{H}}_{\Diffl}^2(\frakg,V)$}.
	
	Conversely, given two 2-cocycles $(\psi_1,\chi_1)$ and $(\psi_2,\chi_2)$, we can construct two abelian extensions $(\frakg\oplus V,[\cdot, \cdot]_{\psi_1},\dd_{\chi_1})$ and  $(\frakg\oplus V,[\cdot, \cdot]_{\psi_2},\dd_{\chi_2})$ via equalities~\eqref{eq:mul} and \eqref{eq:dif}. If they represent the same cohomological class {in $\tilde{\rm{H}}_{\Diffl}^2(\frakg,V)$}, then there exists a linear map $\phi:\frakg\to V$ such that $$(\psi_1,\chi_1)=(\psi_2,\chi_2)+\partial^1_{\Diffl}(\phi).$$ Define $\zeta:\frakg\oplus V\rar \frakg\oplus V$ by
	\[\zeta(x,v):=x+\phi(x)+v.\]
	Then $\zeta$ is an isomorphism of these two abelian extensions.
\end{proof}

\begin{remark}
	In particular,   any vector space $V$ with linear endomorphism $\dd_V$  can serve as a trivial representation of $(\frakg,\dd_\frakg)$. In this situation,  central extensions  of $(\frakg,\dd_\frakg)$ by $(V,\dd_V)$  are classified by the second cohomology group ${\rm{H}}_{\Diffl}^2(\frakg,V)$ of $(\frakg,\dd_\frakg)$ with the coefficient in the trivial representation $(V,\dd_V)$. Note that for a trivial representation $(V,\dd_V)$, since $\partial^1_\DO v=0$ for all $v\in V$, we have $${\rm{H}}_{\Diffl}^2(\frakg,V)=\tilde{\rm{H}}_{\Diffl}^2(\frakg,V).$$
\end{remark}

\section{Deformations of differential Lie algebras}\label{sec:def}

In this section, we study formal deformations  of a differential Lie algebra of weight $\lambda$. In particular, we show that if the second cohomology group $\tilde{\rm{H}}^2_{\Diffl}(\frakg,\frakg)=0$, then the differential Lie algebra  $(\frakg,\dd_\frakg)$ is rigid.

Let $(\frakg,\dd_\frakg)$ be a  differential Lie algebra. Denote by $\mu_\frakg$ the multiplication of $\frakg$.
Consider the 1-parameterized family
$$\mu_t=\sum_{i=0}^{\infty} \mu_i t^i, \, \, \mu_i\in \rmC^2_{\Alg}(\frakg, \frakg),\quad
d_t=\sum_{i=0}^{\infty} d_i t^i, \, \, d_i\in \rmC^1_\diffo(\frakg, \frakg).$$

\begin{defn}
	A {\bf 1-parameter formal deformation} of a differential Lie algebra $(\frakg, \dd_\frakg)$ is a pair $(\mu_t, d_t)$,  which endows the $\bk[[t]]$-module $(\frakg[[t]], \mu_t, d_t)$ with a  differential Lie algebra structure such that $(\mu_0, d_0)=(\mu_\frakg, \dd_\frakg)$.
\end{defn}

The pair $(\mu_t, d_t)$ generates a 1-parameter formal deformation  of the differential Lie algebra $(\frakg, \dd_\frakg)$ if and only if for all $x, y, z\in \frakg$, the following equalities hold:
\begin{eqnarray}\label{equation: ass nonexpanded}
	0 &=&  \mu_t(x, \mu_t(y, z))+\mu_t(y, \mu_t(z, x))+\mu_t(z, \mu_t(x, y)),\\
	\label{equation: derivation nonexpanded}
	d_t(\mu_t(x, y))&=&\mu_t(d_t(x), y)+\mu_t(x, d_t(y))+\lambda \mu_t(d_t(x),  d_t(y)).
\end{eqnarray}
Expanding these equations and collecting coefficients of $t^n$, we see that Eqs.~\eqref{equation: ass nonexpanded} and \eqref{equation: derivation nonexpanded} are equivalent to the systems of equations: for each $n\geq 0$,
\begin{eqnarray}\label{equation: ass}
	0&=&\sum_{i=0}^n \mu_i(x, \mu_{n-i}(y, z))+\mu_i(y, \mu_{n-i}(z, x))+\mu_i(z, \mu_{n-i}(x, y)),\\
	\label{df}
	\sum_{k,l\geq0\atop k+l=n}d_l\mu_k(x,y)&=&\sum_{k,l\geq0\atop k+l=n}\left(\mu_k(d_l(x),y)+\mu_k(x,d_l(y))\right)+\lambda\sum_{k,l,m\geq0\atop k+l+m=n}\mu_k(d_l(x),d_m(y)).
\end{eqnarray}

\begin{remark}For $n=0$, Eq.~\eqref{equation: ass} is equivalent to the Jabobi identity of $\mu_\frakg$, and Eq.~\eqref{df} is equivalent to the fact that $\dd_\frakg$ is a $\lambda$-derivation. For $n=1$, Eq.~\eqref{equation: ass} has the form:
\begin{equation}\label{equation9:n=1}
	\begin{aligned}
		0&=\mu_\frakg(x,\mu_1(y,z))+\mu_\frakg(y,\mu_1(z,x))+\mu_\frakg(z,\mu_1(x,y))\\
	&\quad+\mu_1(x,\mu_\frakg(y,z))+\mu_1(y,\mu_\frakg(z,x))+\mu_1(z,\mu_\frakg(x,y)).	
 \end{aligned}
\end{equation}
And for $n=1$, Eq.~\eqref{df} has the form:
\begin{equation}\label{equation10:n=1}
	\begin{aligned}
		d_1\mu_\frakg(x,y)+\dd_\frakg\mu_1(x,y)
		&=\mu_1(d_\frakg(x),y)+\mu_1(x,\dd_\frakg(y))+ \mu_\frakg(d_1(x),y)+ \mu_\frakg(x,d_1(y))   \\
		&\quad+\lambda\mu_1(\dd_\frakg(x),\dd_\frakg(y))+ \lambda\mu_\frakg(d_1(x),\dd_\frakg(y))+ \lambda\mu_\frakg(\dd_\frakg(x),d_1(y)).
	\end{aligned}
\end{equation}
\end{remark}

\begin{prop}\label{prop:fddco}
	Let $(\frakg[[t]], \mu_t, d_t)$ be a $1$-parameter formal deformation of a differential Lie algebra $(\frakg,\dd_\frakg)$. Then $(\mu_1, d_1)$ is a 2-cocycle of the differential Lie algebra $(\frakg,\dd_\frakg)$ with the coefficient  in the adjoint representation $\frakg_\ad$.	
\end{prop}
\begin{proof}
	For $n =1$, Eq.~\eqref{equation9:n=1} is equivalent to $\partial_\alg^2 \mu_1=0$, and Eq.~\eqref{equation10:n=1} is equivalent to Eq.~\eqref{2-cocycle}, that is $$\partial^1_\DO d_1+\delta^2 \mu_1=0.$$
	Thus  $(\mu_1,d_1)$ is a 2-cocycle by Remark~\ref{ex:cocycle}.	
\end{proof}

If $\mu_t=\mu_\frakg$ in the above $1$-parameter formal deformation of the differential  Lie algebra $(\frakg,\dd_\frakg)$, we obtain a $1$-parameter formal deformation of the differential operator $\dd_\frakg$. Consequently, we have

\begin{coro}\label{coro:fdo}
	Let $ d_t $ be a $1$-parameter formal deformation of the differential operator  $ \dd_\frakg $.  Then $d_1$ is a 1-cocycle of the differential operator  $\dd_\frakg$ with coefficients  in the adjoint representation $\frakg_\ad$.		
\end{coro}
\begin{proof}
	In the special case when $n =1$, Eq.~\eqref{df} is equivalent to $\partial^1_\DO d_1=0$, which implies that  $d_1$ is a 1-cocycle of the differential operator  $\dd_\frakg$ with coefficients in the adjoint representation $\frakg_\ad$.	
\end{proof}

\begin{defn}
	The $2$-cocycle $(\mu_1,d_1)$ is called the {\bf infinitesimal} of the $1$-parameter formal deformation $(\frakg[[t]],\mu_t,d_t)$ of $(\frakg,\dd_\frakg)$.
\end{defn}

\begin{defn}
	Let $(\frakg[[t]],\mu_t,d_t)$ and $(\frakg[[t]],\bar{\mu}_t,\bar{d}_t)$ be $1$-parameter formal deformations of $(\frakg,\dd_\frakg)$. A
	{\bf formal isomorphism} from $(\frakg[[t]],\bar{\mu}_t,\bar{d}_t)$ to $(\frakg[[t]],\mu_t,d_t)$ is a power series $\Phi_t=\sum_{i\geq0}\phi_it^i:\frakg[[t]]\lon \frakg[[t]]$, where $\phi_i:\frakg\to \frakg$ are linear maps with $\phi_0=\Id_\frakg$, such that
	\begin{eqnarray*}
		\Phi_t\circ\bar{\mu}_t&=& \mu_t\circ(\Phi_t\times\Phi_t),\\
		\Phi_t\circ\bar{d}_t&=&d_t\circ\Phi_t.
	\end{eqnarray*}

	Two $1$-parameter formal deformations $(\frakg[[t]],\mu_t,d_t)$ and $(\frakg[[t]],\bar{\mu}_t,\bar{d}_t)$ are said to be {\bf equivalent} if  there exists a formal isomorphism $\Phi_t:(\frakg[[t]],\bar{\mu}_t,\bar{d}_t) \lon (\frakg[[t]],\mu_t,d_t)$.
\end{defn}

\begin{theorem}
	The infinitesimals of two equivalent $1$-parameter formal deformations of $(\frakg,\dd_\frakg)$ are in the same cohomology class {in $\tilde{\rm{H}}_{\Diffl}^2(\frakg,\frakg)$}.
\end{theorem}
\begin{proof}
	Let $\Phi_t:(\frakg[[t]],\bar{\mu}_t,\bar{d}_t)\lon(\frakg[[t]],\mu_t,d_t)$ be a formal isomorphism. For all $x,y\in \frakg$, we have
	\begin{eqnarray*}
		\Phi_t\circ\bar{\mu}_t(x,y)&=& \mu_t\circ(\Phi_t\times\Phi_t)(x,y),\\
		\Phi_t\circ\bar{d}_t(x)&=& d_t\circ\Phi_t (x).
	\end{eqnarray*}
	Expanding the above identities and comparing the coefficients of $t$, we obtain
	\begin{eqnarray*}
		\bar{\mu}_1(x,y)&=&\mu_1(x,y)+[\phi_1(x), y]+[x, \phi_1(y)]-\phi_1([x, y]),\\
		\bar{d}_1(x)&=&d_1(x)+\dd_\frakg(\phi_1(x))-\phi_1(\dd_\frakg(x)).
	\end{eqnarray*}
	Thus, we have $$(\bar{\mu}_1,\bar{d}_1)=(\mu_1,d_1)+\partial_{\Diffl}^1(\phi_1),$$ which implies that $[(\bar{\mu}_1,\bar{d}_1)]=[(\mu_1,d_1)]$ in $\tilde{\rm{H}}^2_{\Diffl}(\frakg,\frakg)$.
\end{proof}

Given any differential Lie algebra $(\frakg,\dd_\frakg)$, interpret
$\mu_\frakg$ and $\dd_\frakg$ as the formal power series
$\mu_t$ and $d_t$ with $\mu_i=\delta_{i,0}\mu_\frakg$ and $d_i=\delta_{i,0}\dd_\frakg$ respectively for all $i\geq0$, where $\delta_{i,0}$ is the Kronecker sign. Then $(\frakg[[t]],\mu_\frakg,\dd_\frakg)$ is a 1-parameter formal deformation of $(\frakg, \dd_\frakg)$.
\begin{defn}
	A $1$-parameter formal deformation $(\frakg[[t]],\mu_t,d_t)$ of $(\frakg,\dd_\frakg)$ is said to be {\bf trivial} if it is equivalent to the  deformation $(\frakg[[t]],\mu_\frakg,\dd_\frakg)$, that is, there exists  $\Phi_t=\sum_{i\geq0}\phi_it^i:\frakg[[t]]\lon \frakg[[t]]$, where  $\phi_i:\frakg\to \frakg$ are linear maps with $\phi_0=\Id_\frakg$, such that
	\begin{eqnarray*}
		\Phi_t\circ\mu_t&=&\mu_\frakg\circ(\Phi_t\times\Phi_t),\\
		\Phi_t \circ d_t&=&\dd_\frakg\circ\Phi_t.
	\end{eqnarray*}
\end{defn}

\begin{defn}
	A differential Lie algebra $(\frakg,\dd_\frakg)$ is said to be {\bf rigid} if every $1$-parameter formal deformation  is trivial.
\end{defn}

\begin{theorem}
	If $\tilde{\rm{H}}^2_{\Diffl}(\frakg,\frakg_\ad)=0$, then the differential Lie algebra  $(\frakg,\dd_\frakg)$ is rigid.
\end{theorem}
\begin{proof}
	Let $(\frakg[[t]],\mu_t,d_t)$ be a $1$-parameter formal deformation of $(\frakg,\dd_\frakg)$. By Proposition ~\ref{prop:fddco},   $(\mu_1,d_1)$ is a 2-cocycle. By $\tilde{\rm{H}}^2_{\Diffl}(\frakg,\frakg_\ad)=0$, there exists a 1-cochain $\phi_1\in  \rmC^1_\alg(\frakg,\frakg)$ such that
	\begin{eqnarray}
		\label{rigid}(\mu_1,d_1)=-\partial_{\Diffl}^1(\phi_1).
	\end{eqnarray}
	Then setting $\Phi_t=\Id_\frakg+\phi_1 t$, we have a 1-parameter formal deformation $(\frakg[[t]],\bar{\mu}_t,\bar{d}_t)$, where
	\begin{eqnarray*}
		\bar{\mu}_t(x,y)&=&\big(\Phi_t^{-1}\circ \mu_t\circ(\Phi_t\times\Phi_t)\big)(x,y),\\
		\bar{d}_t(x)&=&\big(\Phi_t^{-1}\circ d_t\circ\Phi_t\big)(x).
	\end{eqnarray*}
	Thus, $(\frakg[[t]],\bar{\mu}_t,\bar{d}_t)$ is equivalent to $(\frakg[[t]],\mu_t,d_t)$. Moreover, we have
	\begin{eqnarray*}
		\bar{\mu}_t(x,y)&=&(\Id_\frakg-\phi_1t+\phi_1^2t^{2}+\cdots+(-1)^i\phi_1^it^{i}+\cdots)(\mu_t(x+\phi_1(x)t,y+\phi_1(y)t)),\\
		\bar{d}_t(x)&=&(\Id_\frakg-\phi_1t+\phi_1^2t^{2}+\cdots+(-1)^i\phi_1^it^{i}+\cdots)(d_t(x+\phi_1(x)t)).
	\end{eqnarray*}
	Therefore,
	\begin{eqnarray*}
		\bar{\mu}_t(x,y)&=&[x, y]+(\mu_1(x,y)+[x, \phi_1(y)]+[\phi_1(x), y]-\phi_1([x, y]))t+\bar{\mu}_{2}(x,y)t^{2}+\cdots,\\
		\bar{d}_t(x)&=&d_A(x)+(d_A(\phi_1(x))+d_1(x)-\phi_1(\dd_\frakg(x)))t+\bar{d}_{2}(x)t^{2}+\cdots.
	\end{eqnarray*}
	By Eq.~\eqref{rigid}, we have
	\begin{eqnarray*}
		\bar{\mu}_t(x,y)&=&[x, y]+\bar{\mu}_{2}(x,y)t^{2}+\cdots,\\
		\bar{d}_t(x)&=&\dd_\frakg(x)+\bar{d}_{2}(x)t^{2}+\cdots.
	\end{eqnarray*}
	Then by repeating the argument, we can show that $(\frakg[[t]],\mu_t,d_t)$ is equivalent to $(\frakg[[t]],\mu_\frakg,\dd_\frakg)$. Thus, $(\frakg,\dd_\frakg)$ is rigid.
\end{proof}
\bigskip

\bigskip
\section{$L_\infty[1]$-structure for (relative and absolute) differential Lie algebras}\label{sec:L_infty}

In this section, we study  the $L_\infty[1]$-structure for   differential Lie algebras of weight $\lambda$.
In order to deal with absolute differential Lie algebras, we introduce a generalised version of derived bracket technique; moreover, to consider  the weight  case, we  incorporate the weight $\lambda$ into the statements.  These are the main differences of this paper with \cite{PSTZ21, JS23}, since the latter papers consider   difference operators, say, relative differential operators of weight $1$.

\medskip

\subsection{$L_\infty[1]$-algebras}\ \label{Subsect: Linfinity algebras}

In this subsection, we  recall some preliminaries on $L_\infty[1]$-algebras.

\begin{defn}\label{Def:L[1]-infty}
	An \textbf{$L_\infty[1]$-algebra} is a graded vector space  $\frakL=\bigoplus\limits_{i\in\mathbb{N}}\frakL^i$   endowed with   a family of graded linear maps $l_n:\frakL^{\ot n}\rightarrow \frakL, n\geq 1$ of degree $1$ satisfying  the following equations:
	for arbitrary  $n\geq 1$,  $ \sigma\in S_n$ and $x_1,\dots, x_n\in \frakL$,
	\begin{enumerate}
		\item[(i)](graded symmetry)
\begin{equation*} \label{graded sym}
l_n(x_{\sigma(1)},\dots,x_{\sigma(n)})=\varepsilon(\sigma)l_n(x_1,\dots,x_n),
\end{equation*}

		\item[(ii)](generalised Jacobi identity)
		\begin{equation*}\label{graded Jacobi}
\sum_{i=1}^n\sum_{\sigma\in \Sh(i,n-i)}\varepsilon(\sigma)l_{n-i+1}(l_i(x_{\sigma(1)},\dots,x_{\sigma(i)}),x_{\sigma(i+1)},\dots,x_{\sigma(n)})=0.
\end{equation*}
	\end{enumerate}
\end{defn}

\begin{remark} \label{Rem: L[1]-infinity for small n}   Let us consider the generalised Jacobi identity for   $n\leq 3$ with the assumption of  generalised  symmetry.
	
	\begin{enumerate}
		\item[(i)]  For $n=1$,  then $l_1 \circ l_1 =0$, that is,  $l_1 $ is a differential.

		\item[(ii)] For $n=2$, then $l_1\circ l_2 +l_2 \circ (l_1 \ot\Id+\Id\ot l_1)=0$, that is , $l_1$ is a derivation with respect to $l_2$.

		\item[(iii)] For $n=3$ and arbitrary homogeneous elements $x_1, x_2, x_3\in L$, we have
		$$\begin{array}{ll} &l_2 (l_2 (x_1\ot x_2)\ot x_3)+(-1)^{|x_1|(|x_2|+|x_3|)} l_2 (l_2 (x_2\ot x_3)\ot x_1)+
			(-1)^{|x_3|(|x_1|+|x_2|)} l_2 (l_2 (x_3\ot x_1)\ot x_2)
			\\
			=&-\Big(l_1 (l_3 (x_1\ot x_2\ot x_3))+ l_3 (l_1  (x_1)\ot x_2\ot x_3 )+(-1)^{|x_1|} l_3 (x_1\ot l_1  (x_2)\ot x_3 )+\\
			&(-1)^{|x_1|+|x_2|} l_3 (x_1\ot x_2\ot l_1  (x_3) )\Big),\end{array}$$
		 that is, $l_2$ satisfies the   Jacobi identity up to homotopy.
	\end{enumerate}
	
	
\end{remark}



\begin{defn}
A \textbf{Maurer-Cartan element} of an $L_\infty[1]$-algebra $(\frakL,\{l_n \}_{n\geq1})$ is  an element $\alpha\in \frakL^{0}$   satisfying the Maurer-Cartan equation:
	\begin{eqnarray*}\label{Eq: mc-equation[1]}\sum_{n=1}^\infty\frac{1}{n!} l_n (\alpha^{\ot n})=0,\end{eqnarray*}
	whenever this infinite sum exists.  Denote $\mathcal{MC}(\frakL):=\{\mbox{Maurer-Cartan elements of}~ \frakL\}$.
\end{defn}
%

\begin{prop}\cite[Twisting procedure] {Get09} \label{Prop: twist-L-infty[1]}
Let  $\alpha$ be a Maurer-Cartan element of $L_\infty[1]$-algebra $\frakL$,  	The twisted $L_\infty[1]$-algebra  is given by $l_n ^{\alpha}: \frakL^{\ot n}\rightarrow \frakL$ which is defined as follows$\colon$
	\begin{eqnarray*}\label{Eq: twisted L[1] infinity algebra} l^\alpha_n(x_1\ot \cdots\ot x_n)=\sum_{i=0}^\infty\frac{1}{i!}l_{n+i} (\alpha^{\ot i}\ot x_1\ot \cdots\ot x_n),\ \forall x_1, \dots, x_n\in \frakL,\end{eqnarray*}
	whenever these infinite sums exist.
\end{prop}

The following simple observation with immediate proof  will be very useful while considering differential operators of a given  weight $\lambda$.
\begin{prop}\label{Prop:  lambda twisted infinity algebra}
Let $\frakL=(\frakL,\{l_n \}_{n\geq1})$ be an $L_\infty[1]$-algebra and   $\lambda\in \bfk$.
Consider $l_n': \frakL^{\otimes n}\to \frakL, n\geq 1$ given by
\begin{eqnarray*}\label{Eq:  lambda twisted infinity algebra} l'_n(x_1\ot \cdots\ot x_n)= \lambda^{n-1}l_{n} ( x_1\ot \cdots\ot x_n),\ \forall x_1, \dots, x_n\in \frakL,\end{eqnarray*}
then $(\frakL,\{l_n' \}_{n\geq1})$ is also an $L_\infty[1]$-algebra.

If $l_1=0$, then imposing $l_1'=0$ and $l_n'=\lambda^{n-2}l_n$ for $n\geq 2$ also gives a new $L_\infty[1]$-structure on $\frakL$.
\end{prop}
\medskip
 \subsection{A generalised  version of derived bracket technique}\ \label{Subsect: derived bracket technique}

Now we introduce a generalised version of the derived bracket technique  invented by   Voronov   \cite{Vor05a, Vor05b}. Since we consider
differential Lie algebras with weight $\lambda\in \bfk$, we
 insert $\lambda$ in the statements whenever needed by Proposition~\ref{Prop:  lambda twisted infinity algebra}, thus modify corresponding results in \cite{CC22, JS23}.

\begin{defn}
A \textbf{generalised V-data} is a septuple  $(\frakL, \frakm,\iota_\frakm, \fraka, \iota_\fraka, P,\Delta)$, where
\begin{itemize}
\item[$(1)$] $(\frakL,[-,-])$ is a graded Lie algebra,
\item[$(2)$] $(\frakm, [-,-]_\frakm)$ is a graded Lie  algebra together with an injective linear  map $\iota_\frakm: \frakm\to \frakL$ which is a    homomorphism of  graded Lie algebras,
\item[$(3)$] $\fraka$ is an abelian graded Lie  algebra endowed  with an injective linear  map $\iota_\fraka: \fraka\to \frakL$ which is a    homomorphism of  graded Lie algebras,
\item[$(4)$] $P:\frakL\lon\fraka$ is a linear map such that   $P\circ \iota_\fraka=\Id_\fraka$ and $\Ker(P)$ is a graded Lie subalgebra of $\frakL$,
\item[$(5)$] $\Delta\in \mathrm{Ker}(P)^1$ satisfying $[\Delta,\Delta]=0$ and  $[\Delta,\iota_\frakm(\frakm)]\subset \iota_\frakm(\frakm)$.
\end{itemize}
\end{defn}

\begin{theorem}\label{theo:V-data}
Let $(\frakL, \frakm,\iota_\frakm, \fraka, \iota_\fraka, P,\Delta)$ be a generalised V-datum and $\lambda\in\bfk$.
Then  the graded vector space
$s \frakL\oplus \fraka$  has  an $L_\infty[1]$-algebra structure  which is defined  as follows$\colon$
$$\begin{array}{rcl}
  l_1(sf)     &  =    & (-s [\Delta,f], P(f)),\\
  l_1(\xi)     &  =    & P[\Delta,\iota_{\fraka}(\xi)],\\
    l_2(sf,sg)   &  =    &(-1)^{|f|}\lambda s [f,g] , \\
    l_i(sf,\xi_1,\cdots,\xi_{i-1})   &  =    &\lambda^{i-1} P[\cdots[f,\iota_{\fraka}(\xi_1)] ,\cdots,\iota_{\fraka}(\xi_{i-1})] ,  i\geq2,\\
    l_i(\xi_1,\cdots,\xi_{i})   &  =    &\lambda^{i-1} P[\cdots[\Delta,\iota_{\fraka}(\xi_1)] ,\cdots,\iota_{\fraka}(\xi_{i})] ,  i\geq2,
\end{array}$$
for homogeneous elements $f,g\in\frakL$, $\xi,\xi_1,\cdots,\xi_i\in\fraka$ and all other components  of   $\{l_i\}_{i=1}^{+\infty}$ vanish.

Similarly, the graded vector space
$s \frakm\oplus \fraka$  has  also an $L_\infty[1]$-algebra structure  which is defined  as follows$\colon$
$$\begin{array}{rcl}
	l_1(sf)     &  =    & (-s \iota^{-1}_\frakm[\Delta,\iota_\frakm(f)], P\iota_\frakm(f)),\\
	l_1(\xi)     &  =    & P[\Delta,\iota_{\fraka}(\xi)],\\
	l_2(sf,sg)   &  =    &(-1)^{|f|}\lambda s [f,g]_\frakm , \\
	l_i(sf,\xi_1,\cdots,\xi_{i-1})   &  =    &\lambda^{i-1} P[\cdots[\iota_\frakm(f),\iota_{\fraka}(\xi_1)] ,\cdots,\iota_{\fraka}(\xi_{i-1})] ,  i\geq2,\\
	l_i(\xi_1,\cdots,\xi_{i})   &  =    &\lambda^{i-1} P[\cdots[\Delta,\iota_{\fraka}(\xi_1)] ,\cdots,\iota_{\fraka}(\xi_{i})] ,  i\geq2,
\end{array}$$
for homogeneous elements $f,g\in\frakm$, $\xi,\xi_1,\cdots,\xi_i\in\fraka$ and all other components of $\{l_i\}_{i=1}^{+\infty}$  vanish.

Moreover, there exists an injective  $L_\infty[1]$-algebra homomorphism $\iota:s \frak{M}\oplus \fraka\lon s \frakL\oplus \fraka$ induced by $\iota_\frakm$.
\end{theorem}

%
\begin{proof}
Obviously, the quadruple $(\frakL,\mathfrak{\iota_{A}(A)},\iota_{\fraka}P,\Delta)$ is a V-data in the sense of Voronov
\cite{Vor05a}, by  \cite[Section 3]{Vor05a},  there is  an $L_\infty[1]$-algebra $\{l_i\}_{i=1}^{+\infty}$ on $s \frakL\oplus\mathfrak{\iota_A{(A)}}$, where
for homogeneous elements $f,g\in\frakL$, $\xi, \xi_1,\cdots, \xi_i\in\fraka$,
$$\begin{array}{rcl}
  l_1(sf)     &  =    & (-s[\Delta,f],\iota_{\fraka}P(f)),\\
  l_1(\iota_{\fraka}(\xi))     &  =    & \iota_{\fraka}P([\Delta,\iota_{\fraka}(\xi)]),\\
    l_2(sf,sg)   &  =    &(-1)^{|f|}s[f,g], \\
    l_i(sf,\iota_{\fraka}(\xi_1),\cdots,\iota_{\fraka}(\xi_{i-1}))   &  =    & \iota_{\fraka}P[\cdots[f,\iota_{\fraka}(\xi_1)],\cdots,\iota_{\fraka}(\xi_{i-1})],  i\geq2,\\
    l_i(\iota_{\fraka}(\xi_1),\cdots,\iota_{\fraka}(\xi_{i}))   &  =    & \iota_{\fraka}P[\cdots[\Delta,\iota_{\fraka}(\xi_1)],\cdots,\iota_{\fraka}(\xi_{i})],  i\geq2,
\end{array}$$
 and  all other components of $\{l_i\}_{i=1}^{+\infty}$  vanish.
 A similar  $L_\infty[1]$-algebra  structure $\{l_i\}_{i=1}^{+\infty}$ on $s\iota_\frakm (\frakm)\oplus\mathfrak{\iota_A{(A)}}$ exists,
 as well as   an inclusion $ s \iota_\frakm (\frakm)\oplus\mathfrak{\iota_A{(A)}} \hookrightarrow s \frakL\oplus\mathfrak{\iota_A{(A)}}$ of  $L_\infty[1]$-algebras.

By  Proposition~\ref{Prop:  lambda twisted infinity algebra}, we could insert $\lambda$ into the construction of the higher Lie brackets.
 The $L_\infty[1]$-algebra  structure  on  $s \frakL\oplus\mathfrak{\iota_A{(A)}}$ are given by:
for homogeneous elements $f,g\in\frakL$, $\xi, \xi_1,\cdots,$ $\xi_i\in\fraka$,
$$\begin{array}{rcl}
  l_1(sf)     &  =    & (-s[\Delta,f],\iota_{\fraka}P(f)),\\
  l_1(\iota_{\fraka}(\xi))     &  =    & \iota_{\fraka}P([\Delta,\iota_{\fraka}(\xi)]),\\
    l_2(sf,sg)   &  =    &(-1)^{|f|}\lambda s[f,g], \\
    l_i(sf,\iota_{\fraka}(\xi_1),\cdots,\iota_{\fraka}(\xi_{i-1}))   &  =    & \lambda^{i-1} \iota_{\fraka}P[\cdots[f,\iota_{\fraka}(\xi_1)],\cdots,\iota_{\fraka}(\xi_{i-1})],  i\geq2,\\
    l_i(\iota_{\fraka}(\xi_1),\cdots,\iota_{\fraka}(\xi_{i}))   &  =    & \lambda^{i-1}\iota_{\fraka}P[\cdots[\Delta,\iota_{\fraka}(\xi_1)],\cdots,\iota_{\fraka}(\xi_{i})],  i\geq2,
\end{array}$$
 and  all other components of higher Lie brackets vanish.
 The $L_\infty[1]$-algebra  structure  on $s \iota_\frakm (\frakm)\oplus\mathfrak{\iota_A{(A)}}$ are given by:
for homogeneous elements $f,g\in\frakm$, $\xi, \xi_1,\cdots, \xi_i\in\fraka$,
$$\begin{array}{rcl}
  l_1(s\iota_\frakm(f))     &  =    & (-s\iota_\frakm^{-1}([\Delta,\iota_\frakm(f)]),\iota_{\fraka}P \iota_\frakm(f)),\\
  l_1(\iota_{\fraka}(\xi))     &  =    & \iota_{\fraka}P([\Delta,\iota_{\fraka}(\xi)]),\\
    l_2(s\iota_\frakm(f),s\iota_\frakm(g))   &  =    &(-1)^{|f|}\lambda s\iota_\frakm([f,g]_\frakm), \\
    l_i(s\iota_\frakm(f),\iota_{\fraka}(\xi_1),\cdots,\iota_{\fraka}(\xi_{i-1}))   &  =    & \lambda^{i-1} \iota_{\fraka}P[\cdots[\iota_\frakm(f),\iota_{\fraka}(\xi_1)],\cdots,\iota_{\fraka}(\xi_{i-1})],  i\geq2,\\
    l_i(\iota_{\fraka}(\xi_1),\cdots,\iota_{\fraka}(\xi_{i}))   &  =    & \lambda^{i-1}\iota_{\fraka}P[\cdots[\Delta,\iota_{\fraka}(\xi_1)],\cdots,\iota_{\fraka}(\xi_{i})],  i\geq2,
\end{array}$$
 and  all other components of higher Lie brackets  vanish.
 There exists also an inclusion $ s \iota_\frakm (\frakm)\oplus\mathfrak{\iota_A{(A)}} \hookrightarrow s \frakL\oplus\mathfrak{\iota_A{(A)}}$ of  $L_\infty[1]$-algebras.


 Then the theorem holds by the following commutative diagram
 \[\xymatrixcolsep{5pc}\xymatrix{
 	s \frakL\oplus\fraka  \ar[r]_{\cong}^{\left(\begin{matrix}
 			\rm{Id}	&  0\\
 			0	&  \iota_\fraka
 		\end{matrix}\right)} &s \frakL\oplus \iota_\fraka{(\fraka)} \\
 	s \frakm\oplus\fraka \ar[u]^\iota \ar[r]^{\cong}_{\left(\begin{matrix}
 			 \iota_\frakm 	&  0\\
 			0	&  \iota_\fraka
 		\end{matrix}\right)} &s \iota_\frakm{(\frakm)}\oplus\iota_\fraka{(\fraka)}.\ar@{^{(}->}[u]^{\rm{inc}}} \]



\end{proof}

%

\begin{defn} Let $(\frakL, \frakm,\iota_\frakm, \fraka, \iota_\fraka, P,\Delta)$ and  $(\mathfrak{L'}, \mathfrak{M'},\iota_{\frakm'}, \mathfrak{A'}, \iota_{\fraka'},P',\Delta')$ be two   generalised  V-data.
	A \textbf{morphism} between them  is  a triple $f=(f_\frakL, f_\frakm, f_\fraka)$,  where $f_\frakL:\frakL\to\mathfrak{L'}$,
 $f_\frakm: \frakm\to\mathfrak{M'}$ and $f_\fraka: \fraka\to\mathfrak{A'}$  are three  homomorphisms of graded Lie algebras such that
 $f\circ \iota_\frakm= \iota_{\frakm'}\circ f_\frakm$, $f\circ \iota_\fraka= \iota_{\fraka'}\circ f_\fraka$,
 $f_\fraka\circ P=P'\circ f_\frakL$ and $f_\frakL(\Delta)=\Delta'$.
\end{defn}


 The following   result is obvious.
 \begin{prop}\label{prop:L_infty morphism}
 Given a morphism of generalised V-data $$f: (\frakL, \frakm,\iota_\frakm, \fraka, \iota_\fraka, P,\Delta)\to (\mathfrak{L'}, \mathfrak{M'},\iota_{\frakm'}, \mathfrak{A'}, \iota_{\fraka'},P',\Delta'),$$  there exists an $L_\infty[1]$-algebra homomorphism $$\tilde{f}: s\frakm\oplus\fraka \to s\mathfrak{M'}\oplus\mathfrak{A'}$$ induced by $f$.
 \end{prop}

Obviously we have the modified version of Theorem~\ref{theo:V-data} as well as a characterisation of their Maurer-Cartan elements.
\begin{prop}\label{Prop: sM+a is L_inf alg}
Let $(\frakL, \frakm,\iota_\frakm, \fraka, \iota_\fraka, P,\Delta)$ be a generalised V-datum and $\lambda\in\bfk$. If $P\circ \iota_\frakm=0$ and $\Delta =0$, then  there exists  an $L_\infty[1]$-algebra  structure on $s \frakm\oplus{\fraka}$, where the higher Lie bracket $\{l_i\}_{i=1}^{+\infty})$ are given by
$$\begin{array}{rcl}
    l_2(sf,sg)   &  =    &(-1)^{|f|} s[f,g]_\frakm , \\
    l_i(sf,\xi_1,\cdots,\xi_{i-1})   &  =    &\lambda^{i-2} P[\cdots[\iota_\frakm(f),\iota_{{\fraka}}(\xi_1)]_{\frakL},\cdots,\iota_{{\fraka}}(\xi_{i-1})]_{\frakL} ,  i\geq2,\\
\end{array}$$
for homogeneous elements $f,g\in\frakm$, $\xi_1,\cdots,\xi_i\in{\fraka}$ and all other components vanish.

Let $f\in \frakm^1, \xi\in \fraka^0$. The pair $(sf, \xi)$ is a Maurer-Cartan element in the $L_\infty[1]$-algebra    $s \frakm\oplus{\fraka}$
if and only if
$$f\in \mathcal{MC}(\frakm)\ \mathrm{and} \sum_{n=2}^\infty\frac{1}{(n-1)!} \lambda^{n-2} P[\cdots[\iota_\frakm(f),\underbrace{\iota_{{\fraka}}(\xi)]_{\frakL} ,\cdots,\iota_{{\fraka}}(\xi)}_{(n-1)\ \mathrm{times}}]_{\frakL}=0.$$
\end{prop}

\begin{proof}The first statement follows from Proposition~\ref{Prop:  lambda twisted infinity algebra}. For the  second statement,
$(sf, \xi)$ is a Maurer-Cartan element in the $L_\infty[1]$-algebra    $s \frakm\oplus{\fraka}$ if and only if
\begin{equation*}
		\begin{aligned}
			0&=	\sum_{n=1}^\infty \frac{1}{n!}l_n((sf, \xi)^{\otimes n})\\
			&=\frac{1}{2}l_2(sf,sf)+\sum_{n=2}^\infty\frac{1}{(n-1)!}l_n(sf,\underbrace{\xi,\dots,\xi}_{(n-1)\ \mathrm{times}})  \\
			&=-\frac{1}{2}s[f,f]_{\frakm}+ \sum_{n=2}^\infty\frac{1}{(n-1)!} \lambda^{n-2} P[\cdots[\iota_\frakm(f),\underbrace{\iota_{{\fraka}}(\xi)]_{\frakL} ,\cdots,\iota_{{\fraka}}(\xi)}_{(n-1)\ \mathrm{times}}]_{\frakL},
		\end{aligned}
	\end{equation*}
if and only if $$[f,f]_{\frakm}=0\ \mathrm{and}\ \sum_{n=2}^\infty\frac{1}{(n-1)!} \lambda^{n-2} P[\cdots[\iota_\frakm(f),\underbrace{\iota_{{\fraka}}(\xi)]_{\frakL} ,\cdots,\iota_{{\fraka}}(\xi)}_{(n-1)\ \mathrm{times}}]_{\frakL}=0.$$
\end{proof}


\medskip

\subsection{$L_\infty[1]$-structure for relative  differential Lie algebras with weight}\  \label{Subsect: Linifnity structure for relative  differential Lie algebras}

In this subsection, by using the generalised version of derived bracket technique, we found an $L_\infty[1]$-algebra  whose Maurer-Cartan elements   are in bijection with the set of structures of  relative differential Lie algebras of weight $\lambda$, thus generalising the result of Jiang and Sheng \cite{JS23} from weight $1$ case to arbitrary weight cases.

We recall  the classical Nijenhuis-Richardson brackets and basic facts about LieAct triples; for details, see \cite{CC22, JS23}.

Given a vector space $V$, its exterior algebra is $\bigwedge (V):=\bigoplus_{k=0}^\infty \wedge^k V$ and  the reduced version is $\bar{\bigwedge} (V):=\bigoplus_{k=1}^\infty \wedge^k V$. We consider the graded vector space $ \mathrm{Hom}(\bar{\bigwedge} (V),V)$, so for  $f\in\mathrm{Hom}(\wedge^{n+1}V,V)$, its degree is   $n$. The graded space  $\mathrm{Hom}(\bar{\bigwedge} (V),V)$ is a graded Lie algebra \cite{NR66}  under the Nijenhuis-Richardson bracket $[~,~]_{\NR}$ which is defined as follows:
for arbitrary $f\in\mathrm{Hom}(\wedge^{p+1}V,V)$ and $g\in\mathrm{Hom}(\wedge^{q+1}V,V)$,
\begin{equation*}
	[f,g]_{\NR}:=f\bar\circ g-(-1)^{pq}g\bar\circ f,
\end{equation*}
where $f\bar\circ g\in \mathrm{Hom}(\wedge^{p+q+1}V,V)$ is given by
\begin{equation*}
	\begin{aligned}
		f \bar\circ g (x_1,\cdots,x_{p+q+1})
		&=\sum_{\sigma\in \Sh(q+1,p)}\sgn(\sigma)f (g (x_{\sigma(1)},\cdots,x_{\sigma(q+1)}),x_{\sigma(q+2)},\cdots,x_{\sigma(p+q+1)}). \end{aligned}
\end{equation*}

\begin{defn}
	A \textbf{LieAct triple} is a triple $(\frakg,\frakh,\rho)$, where $(\frakg,[~,~]_\frakg)$ and $(\frakh,[~,~]_\frakh)$ are Lie algebras and $\rho:\frakg\lon \mathrm{Der}(\frakh)$ is a homomorphism of Lie algebras.
\end{defn}

\begin{remark}
Given a LieAct triple     $(\frakg,\frakh,\rho)$, the Lie algebra homomorphism $\rho:\frakg\lon \mathrm{Der}(\frakh)$ means that $\frakh$ is a Lie $\frakg$-module and there exists an action $\frakg\otimes\frakh\lon\frakh$ given by $x\cdot u:=\rho(x)(u)$,   subject to the Leibniz rule
	\begin{equation*}
x\cdot[u, v]_\frakh=[x\cdot u, v]_\frakh+[u,x\cdot v]_\frakh, \quad\forall x\in \frakg ~\mbox{and}~ u,v\in\frakh.
\end{equation*}

\end{remark}


%

Given two vector spaces $V$ and $W$, by the isomorphism
$$\begin{array}{rcl}
\varphi:\wedge^n(V\oplus W)	& \longrightarrow     &\bigoplus_{k+l=n, k, l\geq 0}\wedge^kV\otimes \wedge^lW \\
	(v_1+w_1)\wedge\cdots\wedge(v_n+w_n)        &\mapsto      &\sum\limits_{\sigma\in \Sh(k,n-k)}\sgn(\sigma)v_{\sigma(1)}\wedge\cdots\wedge v_{\sigma(k)}\otimes w_{\sigma(k+1)}\wedge\cdots\wedge w_{\sigma(n)},
\end{array}$$
we have an isomorphism $$\mathrm{Hom}(\wedge^n(V\oplus W),V\oplus W)\cong\big(\oplus_{k+l=n}\mathrm{Hom}(\wedge^kV\otimes \wedge^lW,V)\oplus(\oplus_{k+l=n}\mathrm{Hom}(\wedge^kV\otimes \wedge^lW,  W))\big).$$
With this isomorphism in mind, we   have  the following result.
\begin{prop}\rm{(}\cite{CC22}\rm{)}\label{prop: grad sublie alg}
Let $\frakg$ and $\frakh$ be two vector spaces. Let   $$\frakL'=\mathrm{Hom}(\bar{\wedge}(\frakg\oplus\frakh),\frakg\oplus\frakh)$$ and denote $$\mathfrak{M}':=\mathrm{Hom}(\bar{\wedge}\frakg,\frakg)\oplus\mathrm{Hom}(\bar{\wedge}\frakg\otimes\bar\wedge\frakh,\frakh)
\oplus\mathrm{Hom}(\bar{\wedge}\frakh,\frakh).$$ Let  $\iota_{\frakm'}: {\frakm'}\to {\frakL'}$  be the natural injection.
Then the Nijenhuis-Richardson  bracket on  $({\frakL'}, [-,-]_{\NR})$ induces a Lie bracket on $\mathfrak{M}$ such that  $\iota_{\frakm'}: {\frakm'}\to {\frakL'}$ is an injective homomorphism of graded Lie algebras.

Moreover, Maurer-Cartan elements of the graded Lie algebra $\frakM'$  are in bijection with the set of LieAct triple structures on $(\frakg, \frakh)$.

\end{prop}

\begin{defn}[{\cite{CC22, JS23}}]
Let  $(\frakg,\frakh,\rho)$ be a  LieAct triple. A linear map $D:\frakg\lon\frakh$ is called a \textbf{relative differential operator of weight $\lambda$} if the following equality holds:
\begin{equation*}
	D([x, y]_\frakg)=\rho(x)(D(y))-\rho(y)(D(x))+\lambda [D(x), D(y)]_\frakh, \quad\forall x,y\in \frakg.
	\end{equation*}
The quadruple $(\frakg,\frakh,\rho, D)$ is called a \textbf{relative differential Lie algebra of weight} $\lambda$.
\end{defn}

Now we exhibit an   $L_\infty[1]$-algebra whose Maurer-Cartan elements   are in bijection with the set of structures of  relative differential Lie algebras of weight $\lambda$  structures on $(\frakg, \frakh)$, thus generalising the result of Jiang and Sheng \cite{JS23} from weight $1$ to arbitrary weight.

Let $\frakg$ and $\frakh$ be two vector spaces. As above, let
$$\begin{array}{rcl}{\frakL'}&=&\mathrm{Hom}(\bar{\wedge}(\frakg\oplus\frakh),\frakg\oplus\frakh),\\ {\frakm'}&=&\mathrm{Hom}(\bar{\wedge}\frakg,\frakg)\oplus\mathrm{Hom}(\bar{\wedge}\frakg\otimes\bar\wedge\frakh,\frakh)
	\oplus\mathrm{Hom}(\bar{\wedge}\frakh,\frakh),\end{array}$$
 and let $$ {\fraka'}=\mathrm{Hom}(\bar{\wedge}\frakg,\frakh). $$
 Denote  $\iota_{\fraka'}:{\fraka'}\to {\frakL'}$ to be the  natural injection
		and  $P':{\frakL'}\lon{\fraka'}$ to  be the natural surjection. Obviously with the induced Lie bracket,
  ${\fraka'}$ is an abelian graded Lie algebra,  $\iota_{\fraka'}:{\fraka'}\to {\frakL'}$
		 is also a homomorphism  of graded Lie algebras, $\Ker(P')$ is a graded Lie subalgebra of ${\frakL'}$,  $P\circ \iota_{\fraka'}=\Id_{\fraka'}$. Now let $\Delta'=0$.
Then the following result is easy by direct inspection and by Proposition~\ref{Prop: sM+a is L_inf alg}.

\begin{prop}[{Compare with \cite[Proposition 3.7]{JS23}}] \label{prop: V-data for relative and L-infinity}
The data $({\frakL'}, {\frakm'}, \iota_{\frakm'}, {\fraka'}, \iota_{\fraka'}, P', \Delta'=0)$ introduced above  is a generalised V-datum, and
 an $L_\infty[1]$-algebra structure on $s {\frakm'}\oplus{\fraka'}$  is given by
$$\begin{array}{rcl}
	l_2(sf,sg)   &  =    &(-1)^{|f|}  s[f,g]_{\NR}, \\
	l_i(sf,\xi_1,\cdots,\xi_{i-1})   &  =    &\lambda^{i-2} P[\cdots[f,\xi_1]_{\NR},\cdots,\xi_{i-1}]_{\NR},  i\geq2,
\end{array}$$
for homogeneous elements $f,g\in{\frakm'}$, $\xi_1,\cdots,\xi_{i-1}\in{\fraka'}$ and the other components vanish.
\end{prop}

 \begin{theorem}[{Compare with \cite[Theorem 3.8]{JS23}}]

Let $\frakg, \frakh$ be two vector spaces.  Let $\pi\in\mathrm{Hom}(\wedge^2\frakg,\frakg)$, $\rho\in\mathrm{Hom}(\frakg\otimes\frakh,\frakh)$, $\mu\in\mathrm{Hom}(\wedge^2\frakh,\frakh)$ and $D\in\mathrm{Hom}(\frakg,\frakh)$, then
$(s (\pi+\rho+\mu),D)\in \mathcal{MC}(s \frakM'\oplus\fraka')$ if and only if $(\frakg, \frakh, \rho)$ is a  LieAct triple  and $D$ is  a relative differential  operator of weight  $\lambda\in\bfk$.
\end{theorem}

\begin{proof}
Let $\chi:=\pi+\rho+\mu$. By Proposition~\ref{Prop: sM+a is L_inf alg},    $(s\chi,D) \in\mathcal{MC}(s \mathfrak{M'}\oplus\mathfrak{A'})$ if and only if   $\chi\in \mathcal{MC}(\frakm')$ (that is, $(\frakg, \frakh, \rho)$ is a LieAct triple by Proposition~\ref{prop: grad sublie alg})  and
  $\sum\limits_{k=2}^\infty\frac{\lambda^{k-2}}{(k-1)!}P[\dots,[\chi,\underbrace{D]_{\NR},\dots,D}_{(k-1)\ \mathrm{times}}]_{\NR}=0$.

Since $[\chi,D]_{\NR}\in\mathrm{Hom}(\wedge^2\frakg,\frakh)\oplus\mathrm{Hom}( \frakg\ot\frakh,\frakh)$,
$$[[\chi,D]_{\NR},D]_{\NR}\in\mathrm{Hom}(\wedge^2\frakg,\frakh)\ \mathrm{and}\
 [[[\chi,D]_{\NR},D]_{\NR},D]_{\NR}=0,$$  then we obtain
\begin{equation*}
	\begin{aligned}
		0=&	\sum_{k=2}^\infty\frac{\lambda^{k-2}}{(k-1)!}P[\dots,[\chi,\underbrace{D]_{\NR},\dots,D}_{(k-1)\ \mathrm{times}}]_{\NR}\\
		=&P[\chi,D]_{\NR}+\frac{\lambda}{2}P[[\chi,D]_{\NR},D]_{\NR}.
	\end{aligned}
\end{equation*}
  For arbitrary  $ (x, y) \in\wedge^2\frakg$, we have
\begin{equation*}
	\begin{aligned}
		0&=	P[\chi,D]_{\NR}(x,y)+\frac{\lambda}{2}P[[\chi,D]_{\NR},D]_{\NR}(x, y)\\
		&=\rho\bar\circ D(x,y)-D\bar\circ \pi(x,y)+\frac{\lambda}{2}(\mu\bar\circ D)\bar\circ D(x, y)\\
		&=\rho( D(x),y)-\rho( D(y),x)-D([x,y]_\frakg)+\frac{\lambda}{2}\big(\mu\bar\circ D(D(x),y)-\mu\bar\circ D(D(y),x)\big)\\
		&=\rho(x)D(y)-\rho(y) D(x)-D([x,y]_\frakg)+\frac{\lambda}{2}\big(-\mu (D(y),(D(x))+\mu((D(x),D(y))\big)\\
		&=\rho(x)D(y)-\rho(y) D(x)-D([x,y]_\frakg)+\lambda[D(x),D(y)]_\frakh,
	\end{aligned}
\end{equation*}
Hence $D$ is a relative differential  operator of weight  $\lambda$.
\end{proof}

\medskip

\subsection{$L_\infty[1]$-structure for absolute differential Lie algebras}\ \label{Subsect:  Linfinty for differential Lie algebras}

Let $ \frak g$ be a vector space. We consider the graded Lie algebra $$\frakL:=\mathrm{Hom}(\bar{\wedge}(\frakg\oplus\mathfrak{g}),\frakg\oplus\mathfrak{g})$$ endowed with the  Nijenhuis-Richardson bracket $[~,~]_{\NR}$.  For convenience of presentation, we shall write $\frakg'$ for the second $\frakg$ in $\frakg\oplus \frakg$,
that is, $$\frakL:=\mathrm{Hom}(\bar{\wedge}(\frakg\oplus\mathfrak{g}'),\frakg\oplus\mathfrak{g}').$$
Let $$\frakm:=\mathrm{Hom}(\bar{\wedge}\frakg,\frakg)\ \mathrm{and}\
 \fraka:=\Hom(\bar{\wedge}\frakg,\frakg).$$
 Endow $\frakm$ with the Nijenhuis-Richardson bracket and $\fraka$ with the trivial bracket.
Consider two linear maps $\iota_\frakm: \frakm\to \frakL$ and $\iota_\fraka: \fraka\to \frakL$ defined as follows:
For given $f:\wedge^{n+1}\frakg\lon\frakg \in \frakm$,
	 $\iota_\frakm(f)=\sum\limits_{i=0}^{n+1}f_i$
	 where   $f_i:\wedge^{n+1-i}\frakg\ot\wedge^{i}\frakg'\lon\frakg'$, $ 0\leq i\leq n+1$ is given by
	$$f_i(x_1\wedge\dots\wedge x_{n+1-i}\otimes y_{n+2-i}\wedge \dots\wedge y_{n+1}):=f(x_1\wedge\dots\wedge x_{n+1-i}\wedge y_{n+2-i}\wedge \dots\wedge y_{n+1}),  $$  for $x_1\wedge\dots\wedge x_{n+1-i}\otimes y_{n+2-i}\wedge \dots\wedge y_{n+1}\in  \wedge^{n+1-i}\frakg\ot\wedge^{i}\frakg'$;
 $\iota_\fraka$ identifies $\fraka =\Hom(\bar{\wedge}\frakg,\frakg) $ with the subspace $ \Hom(\bar{\wedge}\frakg,\frakg') $ of
$\frakL$.
Let $P:\frakL\to \fraka$ be the natural projection identifying the subspace $ \Hom(\bar{\wedge}\frakg,\frakg') $ with  $\fraka =\Hom(\bar{\wedge}\frakg,\frakg) $.



\begin{prop}\label{prop: lambda V-data-absolute}
	Let $\frakg$   be a vector space and $\lambda\in\bfk$. The datum  $(\frakL, \frakm,\iota_\frakm, \fraka, \iota_\fraka, P,\Delta=0)$ defined above  is a generalised   V-datum.
\end{prop}
\begin{proof}
Obviously, $\iota_\frakm: \frakm\to \frakL$ and $\iota_\fraka: \fraka\to \frakL$ are  injective.	
The only unclear statement is   that $\iota_\frakm: \frakm\to \frakL$ is a homomorphism  of  graded Lie algebras.
	
	In fact, for arbitrary  $f:\wedge^{n+1}\frakg\lon\frakg$ and $g:\wedge^{m+1}\frakg\lon\frakg$ in $\frakm$, we have
	 \begin{equation*}
		\begin{aligned}
			&[\iota(f),\iota(g)]_{\NR}\\
=&	\bigg[\sum_{i=0}^{n+1}f_i,\sum_{j=0}^{m+1}g_j\bigg]_{\NR}\\
			=&[f_0,g_0]_{\NR}+(-1)^{mn+1}\sum_{j=1}^{m}g_j\bar\circ f_0+\sum_{i=1}^{n}f_i\bar\circ g_0+\sum_{i=1}^{n}\sum_{j=1}^{m}[f_i,g_j]_{\NR}
			+\sum_{i=1}^{n}[f_i,g_{m+1}]_{\NR}\\
		&+\sum_{j=1}^{m}[f_{n+1},g_{j}]_{\NR}+[f_{n+1},g_{m+1}]_{\NR}\\
            =& \left(f_0\bar\circ g_0+\sum_{i=1}^{n}f_i\bar\circ g_0+\sum_{i=1}^{n}\sum_{j=1}^{m}f_i\bar\circ  g_j
			+\sum_{i=1}^{n} f_i\bar\circ g_{m+1}+\sum_{j=1}^{m}f_{n+1}\bar\circ g_{j}++f_{n+1}\bar\circ g_{m+1}\right)-\\
&(-1)^{mn}\left(g_0\bar\circ f_0+\sum_{i=1}^{n}g_0\bar \circ f_i+\sum_{j=1}^{m}g_j\bar\circ f_0+\sum_{i=1}^{n}\sum_{j=1}^{m}g_j\bar\circ  f_i
			+\sum_{j=1}^{m}g_j\bar \circ f_{n+1}+\sum_{i=1}^{n} g_{m+1}\bar\circ f_{i}+g_{m+1}\bar\circ f_{n+1}\right)\\
			=&\iota(f\bar\circ g)-(-1)^{mn} \iota(g\bar\circ f)\\
=&\iota([f,g]_{\NR}).
		\end{aligned}
	\end{equation*}
Thus $\iota_\frakm$ is  a graded Lie algebra homomorphism.
\end{proof}

\begin{lemma}\label{key formula}
	Keep the above notations, for arbitrary  $f=\sum\limits_{i=0}^{n+1}f_i\in\Hom(\wedge^{n+1}\frakg,\frakg)$	and $1\leq r\leq n+1$, take  $\xi_i\in\Hom(\wedge^{m_i+1}\frakg,\frakg')$, $1\leq i\leq r$. We have
	$$\begin{aligned}
		(\cdots((f\bar\circ \xi_1)\bar \circ \xi_2)\cdots)\bar\circ\xi_{r}  =    \sum\limits_{\tau\in\Sh(m_r+1,\dots,m_{1}+1,n+1-r)}(-1)^{\sum\limits_{j=1}^{r}(m_1+\cdots+m_{j-1})m_j}f_r  \big(\xi_{r}\otimes\dots\otimes\xi_1\otimes\Id^{\otimes n+1-r}\big)\tau^{-1},
	\end{aligned}$$
	where,  denote $t=n+\sum\limits_{j=1}^{r}m_j+1$, $\tau^{-1}$ acts on $\wedge^tV$ via $\tau^{-1}(x_1,\dots,x_t):=\sgn(\tau) (x_{\tau(1)},\dots,x_{\tau(t)})$.
\end{lemma}
\begin{proof}   In fact, by  definition of $\bar\circ$, we have
	\begin{eqnarray*}
		& &\quad\big((\cdots((f\bar\circ \xi_1)\bar \circ \xi_2)\cdots)\bar\circ\xi_{r}\big)(x_1,\dots,x_t)\\
		& =&\sum_{\sigma_1\in\Sh(m_{r}+1,t-1-m_{r})}\sgn(\sigma_1)\big((\cdots(f_{r}\bar\circ \xi_1)\cdots)\bar\circ\xi_{r-1}\big)\Big(\xi_{r}(x_{\sigma_1(1)},\dots, x_{\sigma_1(m_{r}+1)}),x_{\sigma_1(m_{r}+2)},\dots,x_{\sigma_1(t)}\Big)\\
		& =&\sum_{\substack{\sigma_1\in\Sh(m_{r}+1,t-1-m_{r})\\\sigma_2\in\Sh(m_{r-1}+1,t-1-m_{r}-m_{r-1})}}
		\sgn(\sigma_2)	\sgn(\sigma_1)\\
		& &\big((\cdots(f_{r}\bar\circ \xi_1)\cdots)\bar\circ\xi_{r-2}\big)\Big(\xi_{r-1}(x_{\sigma_1(m_{r}+\sigma_2(1))},\dots,x_{\sigma_1(m_{r}+\sigma_2(m_{r-1}+1))}),
		\xi_{r}(x_{\sigma_1(1)},\dots,x_{\sigma_1(m_{r}+1)}),\\
		&&x_{\sigma_1(m_{r}+\sigma_2(m_{r-1}+3))},\dots,x_{\sigma_1(m_{r}+\sigma_2(t-m_{r}))}\Big)\\
		&=&\cdots\cdots\cdots\\
		& =&\sum_{\begin{subarray}{c}		
				\sigma_1\in\Sh(m_{r}+1,t-1-m_{r}) \\
				\sigma_2\in\Sh(m_{r-1}+1,t-1-m_{r}-m_{r-1})\\
				\dots\\
				\sigma_{r}\in\Sh(m_{1}+1,n)
		\end{subarray}}
		\sgn(\sigma_{r})\cdots\sgn(\sigma_2)\sgn(\sigma_1)\\ &&f_{r}\Big(\xi_{1}(x_{\sigma_1(m_{r}+\sigma_2(m_{r-1}+\dots+\sigma_{r-1}(m_2+\sigma_{r}(1))\cdots))},\dots, x_{\sigma_1(m_{r}+\sigma_2(m_{r-1}+\dots+\sigma_{r-1}(m_2+\sigma_{r}(m_1+1))\cdots))}),\dots,\\
		& &\xi_{r-1}(x_{\sigma_1(m_{r}+\sigma_2(1))},\dots,x_{\sigma_1(m_{r}+\sigma_2(m_{r-1}+1))}),\xi_{r}(x_{\sigma_1(1)},\dots,x_{\sigma_1(m_{r}+1)}),\\
		&&x_{\sigma_1(m_{r}+\sigma_2(m_{r-1}+\dots+\sigma_{r-1}(m_2+\sigma_{r}(m_1+r+1))\cdots))},\dots, x_{\sigma_1(m_{r}+\sigma_2(m_{r-1}+\dots+\sigma_{r-1}(m_2+\sigma_{r}(m_1+n+1))\cdots))}\Big)\\
		& =&\sum_{\begin{subarray}{c}		
				\sigma_1\in\Sh(m_{r}+1,t-1-m_{r}) \\
				\sigma_2\in\Sh(m_{r-1}+1,t-2-m_{r}-m_{r-1})\\
				\dots\\
				\sigma_{r}\in\Sh(m_{1}+1,n)
		\end{subarray}}
		(-1)^{\frac{r(r-1)}{2}   }\sgn(\sigma_1)\sgn(\sigma_2)\cdots\sgn(\sigma_{r})\\
		& &f_{r} \Big(\xi_{r}(x_{\sigma_1(1)},\dots,x_{\sigma_1(m_{r}+1)}),\xi_{r-1}(x_{\sigma_1(m_{r}+\sigma_2(1))},\dots,x_{\sigma_1(m_{r}+\sigma_2(m_{r-1}+1))}),\dots,\\
		&&\xi_{1}(x_{\sigma_1(m_{r}+\sigma_2(m_{r-1}+\dots+\sigma_{r-1}(m_2+\sigma_{r}(1))\cdots))},\dots, x_{\sigma_1(m_{r}+\sigma_2(m_{r-1}+\dots+\sigma_{r-1}(m_2+\sigma_{r}(m_1+1))\cdots))}),\\
		& &x_{\sigma_1(m_{r}+\sigma_2(m_{r-1}+\dots+\sigma_{r-1}(m_2+\sigma_{r}(m_1+r+1))\cdots))},\dots, x_{\sigma_1(m_{r}+\sigma_2(m_{r-1}+\dots+\sigma_{r-1}(m_2+\sigma_{r}(m_1+n+1))\cdots))}\Big)\\
		& =&\sum_{\begin{subarray}{c}		
				\sigma_1\in\Sh(m_{r}+1,t-1-m_{r}) \\
				\sigma_2\in\Sh(m_{r-1}+1,t-2-m_{r}-m_{r-1})\\
				\dots\\
				\sigma_{r}\in\Sh(m_{1}+1,n)
		\end{subarray}}
		(-1)^{\frac{r(r-1)}{2}+\kappa   }\sgn(\sigma_1)\sgn(\sigma_2)\cdots\sgn(\sigma_{r})f_r  \big(\xi_{r}\otimes\dots\otimes\xi_1\otimes\Id^{\otimes n+1-r}\big)\\
		& & \Big(x_{\sigma_1(1)},\dots,x_{\sigma_1(m_{r}+1)}|, x_{\sigma_1(m_{r}+\sigma_2(1))},\dots,x_{\sigma_1(m_{r}+\sigma_2(m_{r-1}+1))}|,\dots,\\
		&&|x_{\sigma_1(m_{r}+\sigma_2(m_{r-1}+\dots+\sigma_{r-1}(m_2+\sigma_{r}(1))\cdots))},\dots, x_{\sigma_1(m_{r}+\sigma_2(m_{r-1}+\dots+\sigma_{r-1}(m_2+\sigma_{r}(m_1+1))\cdots))}|,\\
		& &|x_{\sigma_1(m_{r}+\sigma_2(m_{r-1}+\dots+\sigma_{r-1}(m_2+\sigma_{r}(m_1+r+1))\cdots))},\dots, x_{\sigma_1(m_{r}+\sigma_2(m_{r-1}+\dots+\sigma_{r-1}(m_2+\sigma_{r}(m_1+n+1))\cdots))}\Big)\\
		&&(\mbox{where $\kappa$ is given by Koszul sign})\\
		&=&\sum\limits_{\tau\in\Sh(m_r+1,\dots,m_{1}+1,n+1-r)}(-1)^{\sum\limits_{j=1}^{r}(m_1+\cdots+m_{j-1})m_j}f_r  \big(\xi_{r}\otimes\dots\otimes\xi_1\otimes\Id^{\otimes n+1-r}\big)\tau^{-1}(x_1,\dots,x_t).\end{eqnarray*}
\end{proof}

By Proposition~\ref{Prop: sM+a is L_inf alg}, we have:
\begin{prop}\label{prop: Linfinity for absolute}
	Keep the above notations, there exists an $L_\infty[1]$-algebra structure on  $s \mathfrak{M}\oplus\fraka$, where $l_i$ are given by
		$$
		l_2(sf,sg)      =    (-1)^{|f|} s[f,g]_{\NR}, \ \
		l_2(sf,\xi)     =     [f, \xi]_{\NR}, $$
and for $3\leq i\leq n+2$,
		$$l_i(sf,\xi_1,\cdots,\xi_{i-1})      =   \sum\limits_{\tau\in\Sh(m_{i-1}+1,\dots,m_1+1,n+2-i)}(-1)^{\sum\limits_{j=1}^{i-1}(m_1+\cdots+m_{j-1})m_j}\lambda^{i-2} f  \big(\xi_{i-1}\otimes\dots\otimes\xi_1\otimes\Id^{\otimes n+2-i}\big)\tau^{-1},
	 $$
	for homogeneous elements
	$f\in\Hom(\wedge^{n+1}\frakg,\frakg)\subseteq \frakm$, $g\in\Hom(\wedge^{m+1}\frakg,\frakg)\subseteq \frakm$, $\xi\in \Hom(\wedge^{m+1}\frakg,\frakg)\subseteq \fraka$, 	and $\xi_j\in\Hom(\wedge^{m_j+1}\frakg,\frakg)\subseteq \fraka$, $1\leq j\leq i-1$, and all others components vanish.
\end{prop}

\begin{proof}
  The assertion follows from   Proposition~\ref{prop: lambda V-data-absolute},     Proposition~\ref{Prop: sM+a is L_inf alg},  and  Lemma \ref{key formula}. Indeed, by Propositions~\ref{prop: lambda V-data-absolute} and \ref{Prop: sM+a is L_inf alg},
  $$\begin{array}{rcl}
  	l_2(sf,sg)   &  =    &(-1)^{|f|} s [f,g]_{\NR}, \\
  	l_i(sf,\xi_1,\cdots,\xi_{i-1})   &  =    &\lambda^{i-2} P[\cdots[\iota_\mathfrak{M}(f),\iota_\fraka(\xi_1)]_{\NR},\cdots,\iota_\fraka{(\xi_{i-1})}]_{\NR},  i\geq2.
  \end{array}
$$
On one hand, for $i=2$, we have $$\begin{aligned}
l_2(sf,\xi)&=P[\iota_\frakm(f),\iota_\fraka(\xi)]_{\NR}\\
&=P\bigg(\sum_{i=0}^{n+1}f_i\bar\circ\iota_\fraka(\xi)-(-1)^{|f||\xi|}\iota_\fraka(\xi)\bar\circ\sum_{i=0}^{n+1}f_i\bigg) \\
&=P(\sum_{i=1}^{n+1} f_i\bar\circ \iota_\fraka(\xi)-(-1)^{|f||\xi|}\iota_\fraka(\xi)\bar\circ f_0\bigg)\\
&=f\bar\circ \xi-(-1)^{|f||\xi|}\xi\bar\circ f\\
&=[f, \xi]_{\NR}.
\end{aligned}$$
On the other hand, for $i\geq3$, we have
$$\begin{aligned}
	&\quad P[\cdots[\iota_\frakm(f),\iota_\fraka(\xi_1)]_{\NR},\dots,\iota_\fraka(\xi_{i-1})]_{\NR}\\
	&=P[\cdots[\sum_{j=1}^{n+1}f_{j}\bar\circ \iota_\fraka(\xi_1)-(-1)^{|f||\xi_1|}\iota_\fraka(\xi_1)\bar\circ f_0,\iota_\fraka(\xi_2)]_{\NR},\dots,\iota_\fraka(\xi_{i-1})]_{\NR}\\
	&=P[\cdots[\sum_{j=2}^{n+1}(f_{j}\bar\circ \iota_\fraka(\xi_1))\bar\circ\iota_\fraka(\xi_2), \iota_\fraka(\xi_3)]_{\NR},\dots,\iota_\fraka(\xi_{i-1})]_{\NR}\\
	&=P\big((\cdots\sum_{j=i-1}^{n+1}(f_{j}\bar\circ \iota_\fraka(\xi_1))\cdots)\bar\circ\iota_\fraka(\xi_{i-1})\big)\\
	&=P\big((\cdots(f_{i-1}\bar\circ \iota_\fraka(\xi_1))\bar\circ\cdots)\bar\circ\iota_\fraka(\xi_{i-1})\big)\\
	&=\big(\cdots(f\bar\circ \xi_1)\bar\circ\cdots\big)\bar\circ \xi_{i-1}\\
&=\sum\limits_{\tau\in\Sh(m_{i-1}+1,\dots,m_1+1,n+2-i)}(-1)^{\sum\limits_{j=1}^{i-1}(m_1+\cdots+m_{j-1})m_j}\lambda^{i-2} f  \big(\xi_{i-1}\otimes\dots\otimes\xi_1\otimes\Id^{\otimes n+2-i}\big)\tau^{-1},
\end{aligned}$$
where the last equality follows from Lemma \ref{key formula}.
\end{proof}

\begin{exam}\label{example spcial case} Consider a special case where $n=1, i=3, t=2, m_1=m_2=0$.
	For homogeneous elements
	$\pi\in\Hom(\wedge^{2}\frakg,\frakg)$, $\xi_1=\xi_2=D\in\Hom(\frakg,\frakg)$,    we have
	\begin{equation*}
		\begin{aligned}
	l_3(s\pi,D,D)(x,y)		
			&=\sum_{\tau\in\Sh(1,1)}\lambda \pi  (D\otimes D\big)~\tau^{-1}(x,y)  \\
			&=\lambda\pi(Dx,Dy)-\lambda\pi(Dy,Dx)\\
			&=2\lambda\pi(Dx,Dy).
		\end{aligned}
	\end{equation*}
\end{exam}

\begin{theorem}\label{Thm: MC elements in absolute Linifnity}
	Let $\frakg$ be a vector space.  Let $\pi\in\mathrm{Hom}(\wedge^2\frakg,\frakg)$, $D\in\mathrm{Hom}(\frakg,\mathfrak{g})$, then
	$(s \pi,D)\in \mathcal{MC}(s \mathfrak{M}\oplus\fraka)$ if and only if $(\frakg,\pi,D)$ is  a differential Lie algebra  of weight  $\lambda\in\bfk$.
\end{theorem}
\begin{proof} By Proposition~\ref{Prop: sM+a is L_inf alg},
	 $(s\pi,D) \in\mathcal{MC}(s \mathfrak{M}\oplus\fraka)$ if and only if
$[\pi,\pi]_{\NR}=0$ (that is, $\pi$ is a Lie bracket) and
	\begin{equation*}
		\begin{aligned}
			0 &= \sum_{k=2}^\infty\frac{1}{(k-1)!}l_k(s\pi,\underbrace{D,\dots,D}_{(k-1)\ \mathrm{times}})  \\
			&=  l_2(s\pi,D)+\frac{\lambda}{2}l_3(s\pi,D,D)\\
			&=  \pi\bar\circ D-D\bar\circ \pi+\frac{\lambda}{2}l_3(s\pi,D,D).
		\end{aligned}
	\end{equation*}
	For arbitrary  $(x,y)\in\wedge^2\frakg$, we have
	\begin{equation*}
		\begin{aligned}
			0&=	\big(\pi\bar\circ D-D\bar\circ \pi+\frac{\lambda}{2}l_3(s\pi,D,D)\big)(x,y)\\
			&=\pi\bar\circ D(x,y)-D\bar\circ \pi(x,y)+\lambda\pi(Dx,Dy)\quad  (\mbox{by Example}~\ref{example spcial case})\\
			&=\pi( D(x),y)-\pi( D(y),x)-D([x,y])+\lambda[D(x),D(y)]\\
			&=[D(x),y]-[D(y),x]-D([x,y])+\lambda[D(x),D(y)]\\
			&=[D(x),y]+[x,D(y)]-D([x,y])+\lambda[D(x),D(y)].
		\end{aligned}
	\end{equation*}
	Hence $D$ is a differential operator of weight  $\lambda$.
\end{proof}

\medskip

\subsection{Equivalences between $L_\infty[1]$-structures for absolute and relative differential Lie algebras with weight}\
\label{Subsect: Relative vs absolute}

\subsubsection{From relative to absolute}\

Let $(\frakg, \mu, \dd)$ be a differential Lie algebra of weight $\lambda$.

Let $\frakh=\frakg$, $\rho: \frakg\to \mathrm{Der}(\frakh)$ be the adjoint representation.
Consider $\dd$ as a map $\dd: \frakg\to \frakh$.
Then $(\frakg, \frakh, \rho, \dd)$ is a relative differential Lie algebra of weight $\lambda$.
As seen in Subsection~\ref{Subsect: Linifnity structure for relative  differential Lie algebras},  let
$$\begin{array}{rcl}{\frakL'}&=&\mathrm{Hom}(\bar{\wedge}(\frakg\oplus\frakh),\frakg\oplus\frakh),\\ {\frakm'}&=&\mathrm{Hom}(\bar{\wedge}\frakg,\frakg)\oplus\mathrm{Hom}(\bar{\wedge}\frakg\otimes\bar\wedge\frakh,\frakh)
	\oplus\mathrm{Hom}(\bar{\wedge}\frakh,\frakh),\\
  {\fraka'}&=&\mathrm{Hom}(\bar{\wedge}\frakg,\frakh). \end{array}$$
 Denote  $\iota_{\fraka'}:{\fraka'}\to {\frakL'}$ and $\iota_{\frakm'}:{\frakm'}\to {\frakL'}$ to be the  natural injection
		and  $P:{\frakL'}\lon{\fraka'}$ to  be the natural surjection.
Then $({\frakL'}, {\frakm'}, \iota_{\frakm'}, {\fraka'},\iota_{\fraka'}, P, \Delta'=0)$ is a generalised V-datum.

Recall that  in Subsection~\ref{Subsect:  Linfinty for differential Lie algebras}, let $\frakg'= \frakg$   and let $$\begin{array}{rcl}\frakL&=&\mathrm{Hom}(\bar{\wedge}(\frakg\oplus\mathfrak{g}'),\frakg\oplus\mathfrak{g}'),\\
   \frakm&=&\mathrm{Hom}(\bar{\wedge}\frakg,\frakg),\\
 \fraka&=&\Hom(\bar{\wedge}\frakg,\frakg).\end{array}$$
Let two linear maps $\iota_\frakm: \frakm\to \frakL$ and $\iota_\fraka: \fraka\to \frakL$ be defined as follows:
For given $f:\wedge^{n+1}\frakg\lon\frakg \in \frakm$,
	 $\iota_\frakm(f)=\sum\limits_{i=0}^{n+1}f_i$
	 where   $f_i:\wedge^{n+1-i}\frakg\ot\wedge^{i}\frakg'\lon\frakg'$, $ 0\leq i\leq n+1$ is given by
	$$f_i(x_1\wedge\dots\wedge x_{n+1-i}\otimes y_{n+2-i}\wedge \dots\wedge y_{n+1}):=f(x_1\wedge\dots\wedge x_{n+1-i}\wedge y_{n+2-i}\wedge \dots\wedge y_{n+1}),  $$  for $x_1\wedge\dots\wedge x_{n+1-i}\otimes y_{n+2-i}\wedge \dots\wedge y_{n+1}\in  \wedge^{n+1-i}\frakg\ot\wedge^{i}\frakg'$;
 $\iota_\fraka$ identifies $\fraka =\Hom(\bar{\wedge}\frakg,\frakg) $ with the subspace $ \Hom(\bar{\wedge}\frakg,\frakg') $ of
$\frakL$.
Let $P:\frakL\to \fraka$ be the natural projection identifying the subspace $ \Hom(\bar{\wedge}\frakg,\frakg') $ with  $\fraka =\Hom(\bar{\wedge}\frakg,\frakg) $.
By Proposition~\ref{prop: lambda V-data-absolute},  $(\frakL, \frakm,\iota_\frakm, \fraka, \iota_\fraka, P,\Delta=0)$    is a generalised   V-datum.

By identifying $\frakg'$ with $\frakh$, let $f_\frakL: \frakL\to \frakL'$ and $f_\fraka: \fraka\to\mathfrak{A'}$   be the identity maps,
 let $f_\frakm: \frakm\to\mathfrak{M'}$ be exactly defined as $\iota_\frakm$.

 The following result is clear.
 \begin{prop}The triple $f=(f_\frakL, f_\frakm, f_\fraka)$ is a morphism of generalised V-data from $(\frakL, \frakm,\iota_\frakm, \fraka,$ $ \iota_\fraka, P,\Delta=0)$ to $({\frakL'}, {\frakm'}, \iota_{\frakm'}, {\fraka'},\iota_{\fraka'}, P, \Delta'=0)$. It induces an injective homomorphism of $L_\infty[1]$-algebras from $s \mathfrak{M}\oplus\fraka$ introduced in Proposition~\ref{prop: Linfinity for absolute} to $s \mathfrak{M}'\oplus\fraka'$ introduced in Proposition~\ref{prop: V-data for relative and L-infinity}.

 \end{prop}

The above result means that  one can deduce the $L_\infty[1]$-structure of absolute differential Lie algebras   from that of relative differential Lie algebras.

\subsubsection{From absolute to relative}\

Let  $(\frakg,\frakh,\rho)$ be  a LieAct triple   and $D:\frakg\lon\frakh$ be a linear map.

By \cite{PSTZ21}, there exists a Lie algebra $\frakg\ltimes_\rho \frakh  $, where the Lie bracket is given by
$$[x + u, y + v]_\ltimes = [x, y]_\frakg + \rho(x)v - \rho(y)u + \lambda[u, v]_\frakh, \forall x, y \in \frakg, u, v \in \frakh.$$
The map $D$ extends to $\tilde{D}: \frakg\ltimes_\rho \frakh \to \frakg\ltimes_\rho \frakh$ by $\tilde{D}(x+u)=D(x)-u,\forall x  \in \frakg, u  \in \frakh$. Furthermore, we have the following observation which seems to be new.
\begin{prop}\label{new obserbation}
$\tilde{D}$ is a differential operator of weight $\lambda$ if and only if $D$ is a relative differential operator of weight $\lambda$.
\end{prop}
\begin{proof}
	According to the definition of $\tilde{D}$, we have
\begin{equation}\label{tildeD1}
	\begin{aligned}
		&\tilde{D}([x + u, y + v]_\ltimes)\\
		=&\tilde{D}( [x, y]_\frakg + \rho(x)v - \rho(y)u + \lambda[u, v]_\frakh)\\
		=&D([x, y]_\frakg)- \rho(x)v + \rho(y)u - \lambda[u, v]_\frakh.
		\end{aligned}	
\end{equation}
On the other hand,
\begin{equation}\label{tildeD2}
	\begin{aligned}
	&[x+u, \tilde{D}(y+v)]_\ltimes+[\tilde{D}(x+u),y+v]_\ltimes+\lambda [\tilde{D}(x+u), \tilde{D}(y+v)]_\ltimes\\
				=&[x+u, D(y)-v]_\ltimes+[ D(x)-u,y+v]_\ltimes+\lambda [D(x)-u, D(y)-v]_\frakh\\
				= & \rho(x)(D(y)-v)+\lambda[u, D(y)-v]_\frakh-\rho(y) (D(x)-u)-\lambda[v, D(x)-u]_\frakh\\
&+\lambda [D(x)-u, D(y)-v]_\frakh\\
				=&\rho(x)(D(y))-\rho(x)(v)+\lambda[u,D(y)]_\frakh- \lambda[u, v]_\frakh-\rho(y)(D(x))+\rho(y)(u)+\lambda[D(x),v]_\frakh\\ &- \lambda[u, v]_\frakh+\lambda [D(x), D(y)]_\frakh- \lambda[D(x),v]_\frakh-\lambda[u,D(y)]_\frakh+ \lambda[u, v]_\frakh.\\
		\end{aligned}	
\end{equation}
Compare the Eq.~(\ref{tildeD1}) with
Eq.~(\ref{tildeD2}), it is easy to see $\tilde{D}$ is a differential operator of weight $\lambda$ if and only if $D$ is a relative differential operator of weight $\lambda$.
\end{proof}
Now let  $(\frakg,\frakh,\rho,D)$ be  a relative differential Lie algebra of weight $\lambda$, then $(\frakg\ltimes_\rho \frakh,[~,~]_\ltimes,\tilde{D})$ is a differential Lie algebra of weight $\lambda$ by Proposition~\ref{new obserbation}.
As seen in Subsection~\ref{Subsect:  Linfinty for differential Lie algebras},  let $(\frakg\ltimes_\rho \frakh)'=\frakg\ltimes_\rho \frakh$ and
$$\begin{array}{rcl}{\frakL}&=&\mathrm{Hom}(\bar{\wedge}(\frakg\ltimes_\rho \frakh\oplus(\frakg\ltimes_\rho \frakh)'),\frakg\ltimes_\rho \frakh\oplus(\frakg\ltimes_\rho \frakh)'),\\ {\frakm}&=&\mathrm{Hom}(\bar{\wedge}(\frakg\ltimes_\rho \frakh),\frakg\ltimes_\rho \frakh),\\
	{\fraka}&=&\mathrm{Hom}(\bar{\wedge}(\frakg\ltimes_\rho \frakh),\frakg\ltimes_\rho \frakh). \end{array}$$
The two linear maps $\iota_\frakm: \frakm\to \frakL$ and $\iota_\fraka: \fraka\to \frakL$ defined as that in Subsection~\ref{Subsect:  Linfinty for differential Lie algebras}. Then $({\frakL}, {\frakm},\iota_\frakm, {\fraka},\iota_\fraka, P, \Delta=0)$ is a generalised V-datum.

Recall in Subsection~\ref{Subsect: Linifnity structure for relative  differential Lie algebras},  let
$$\begin{array}{rcl}{\frakL'}&=&\mathrm{Hom}(\bar{\wedge}(\frakg\oplus\frakh),\frakg\oplus\frakh),\\ {\frakm'}&=&\mathrm{Hom}(\bar{\wedge}\frakg,\frakg)\oplus\mathrm{Hom}(\bar{\wedge}\frakg\otimes\bar\wedge\frakh,\frakh)
	\oplus\mathrm{Hom}(\bar{\wedge}\frakh,\frakh),\\
	{\fraka'}&=&\mathrm{Hom}(\bar{\wedge}\frakg,\frakh). \end{array}$$
Denote  $\iota_{\fraka'}:{\fraka'}\to {\frakL'}$ to be the  natural injection
and  $P:{\frakL'}\lon{\fraka'}$ to  be the natural surjection.
Then $({\frakL'}, {\frakm'},\iota_{\frakm'}, {\fraka'},\iota_{\fraka'}, P, \Delta'=0)$ is a generalised V-datum.

Let $f_{\frakL'}: \frakL'\to \frakL$,
 $f_{\frakm'}: \frakm'\to\mathfrak{M}$ and $f_{\fraka'}: \fraka'\to\mathfrak{A}$ be natural injection maps.

The following result is clear.
\begin{prop}The triple $f=(f_{\frakL'}, f_{\frakm'}, f_{\fraka'})$ is a morphism of generalised V-data from $({\frakL'}, {\frakm'},\iota_{\frakm'},$ ${\fraka'}, \iota_{\fraka'},P, \Delta'=0)$ to $(\frakL, \frakm,\iota_\frakm, \fraka,$ $ \iota_\fraka, P,\Delta=0)$. It induces an injective homomorphism of $L_\infty[1]$-algebras from $s \mathfrak{M}'\oplus\fraka'$ introduced in Proposition~\ref{prop: V-data for relative and L-infinity} to $s \mathfrak{M}\oplus\fraka$ introduced in Proposition~\ref{prop: Linfinity for absolute}.
\end{prop}

The above result means that  one can deduce the $L_\infty[1]$-structure of   relative differential Lie algebras from that of absolute differential Lie algebras.

%
\bigskip

\section{Application of $L_\infty[1]$-structure for differential Lie algebras} \label{Sect: Applications}

In this section, we will derive from the $L_\infty[1]$-structure the cohomology theory of differential Lie algebras of arbitrary weight and the notion of homotopy differential Lie algebras of arbitrary weight.

\subsection{Cohomology of differential Lie algebras from  $L_\infty[1]$-structure}\
\label{Subsect: Cohomology of differential Lie algebras from  Linfinit structure}

Let $\frakg$ be a vector space.  Let $$\frakm:=\mathrm{Hom}(\bar{\wedge}\frakg,\frakg)\ \mathrm{and}\
\fraka:=\Hom(\bar{\wedge}\frakg,\frakg).$$
Recall that we have constructed an $L_\infty[1]$-structure on $s\frakm\oplus\fraka$ in Proposition~\ref{prop: Linfinity for absolute}.

Let $(\frakg, \mu, \dd)$ be a differential Lie algebra.  By Theorem~\ref{Thm: MC elements in absolute Linifnity}, $(s\mu, \dd)$ is a Maurer-Cartan element in the $L_\infty[1]$-algebra $s\frakm\oplus\fraka$. By  Proposition~\ref{Prop: twist-L-infty[1]}, twisting $s\frakm\oplus\fraka$ by  $(s\mu, \dd)$ gives a new $L_\infty[1]$-algebra, whose new differential is denoted by $l_{1}^{(s\mu,\dd)}$.

Consider the cochain complex $\rmC_{\Diffl}^*(\frakg, \frakg_\ad)$  of the differential Lie algebra $(\frakg, \mu, \dd)$ with coefficients in the adjoint representation $\frakg_\ad$.   Note that for each $n\geq  1$,   $$\rmC_{\Diffl}^n(\frakg, \frakg_\ad)=
			\rmC^n_\alg(\frakg,\frakg_\ad)\oplus \rmC^{n-1}_\DO(\frakg,\frakg_\ad)=\mathrm{Hom}( \wedge^n\frakg,\frakg)\oplus \mathrm{Hom}( \wedge^{n-1}\frakg,\frakg)=(s\frakm\oplus\fraka)^{n-2}$$
is exactly the degree $n-2$ part of $s\frakm\oplus\fraka$.




\begin{prop}\label{cochaincomplexad}
	The underlying complex of the twisted  $L_\infty[1]$-algebra  $s\frakm\oplus\fraka$ is exactly the double shift of
the cochain complex $\rmC_{\Diffl}^*(\frakg, \frakg_\ad)$, up to signs.
\end{prop}
\begin{proof}
%
%

It suffices to make explicit the differential $l_1^{(s\mu, \dd)}$.

For  $n\geq 1$, $f\in \mathrm{Hom}( \wedge^n\frakg,\frakg), g\in  \mathrm{Hom}( \wedge^{n-1}\frakg,\frakg)$,
\begin{eqnarray*}
	 l_1^{(s\mu, \dd)}(sf, g)
	&=& \sum_{k=0}^{\infty}\frac{1}{k!}l_{k+1}(\underbrace{(s\mu,\dd),\cdots,(s\mu,\dd)}_{k\ \mathrm{times}}, (sf, g))\\
	&=& l_2((s\mu, \dd),(sf,g))
	 + \sum_{k=2}^{\infty}\frac{1}{k!}l_{k+1}(\underbrace{(s\mu,\dd),\cdots,(s\mu,\dd)}_{k\ \mathrm{times}}, (sf, g))\\
	&=& \big(l_2(s\mu, sf), l_2(s\mu, g)+l_3(s\mu, \dd, g)+l_2(sf, \dd)
	 + \sum_{k=2}^{n}\frac{1}{k!}l_{k+1}(sf, \underbrace{\dd,\cdots,\dd}_{k\ \mathrm{times}}) \big).
\end{eqnarray*}

It is easy to see that
 $l_2(s\mu, sf)=- s[\mu, f]_{\NR}$ is exactly $-s\partial_\Lie^{n}(f)$.

 Let us compute $l_2(s\mu, g)+l_3(s\mu, \dd, g)
	 +l_2(sf, \dd)+ \sum\limits_{k=2}^{n}\frac{1}{k!}l_{k+1}(sf, \underbrace{\dd,\cdots,\dd}_{k\ \mathrm{times}})$.


For   $x_1,\dots,x_{n}\in \frakg$, we have
\begin{eqnarray*}
	&&(l_2(s\mu, g)+l_3(s\mu, \dd, g))(x_1, \dots, x_{n})\\
	&=&[\mu, g]_\NR (x_1, \dots, x_{n})+
	\lambda \sum\limits_{\tau\in\Sh( 1, n-1)} \mu  \big(g\otimes \dd \big)\tau^{-1}(x_1, \dots, x_{n})\\
	&=&\sum_{i=1}^{n}(-1)^{i+n-1}\rho(x_i ) g(x_1,\dots,\hat{x}_i, \dots, x_{n})\\
	&&+\sum_{1\leq i<j\leq n}^{n-1}(-1)^{i+j+n}g([x_i, x_j], x_1,\dots,\hat{x}_i, \dots,\hat{x}_j, \dots,x_{n})\\
	&&+\lambda\sum_{i=1}^{n}(-1)^{i+n-1}\rho(  \dd(x_i)) g(x_1,\dots,\hat{x}_i, \dots, x_{n})\\
	&=& \partial_\DO^{n-1}(g)(x_1, \dots, x_n).
\end{eqnarray*}
On the other hand,
\begin{eqnarray*}
	&&l_2(sf, \dd)+\sum\limits_{k=2}^{n}\frac{1}{k!}l_{k+1}(sf, \underbrace{\dd,\cdots,\dd}_{k\ \mathrm{times}})(x_1, \dots, x_{n})\\
	&=&\big(-\dd\circ f+\sum_{k=1}^{n}\frac{1}{k!}\lambda^{k-1}f\bar\circ\{\underbrace{\dd,\cdots,\dd}_{k\ \mathrm{times}}\}\big)(x_1, \dots, x_{n})\\
	&=&-\dd(f(x_1,\cdots,x_n))+\sum_{k=1}^{n}\frac{1}{k!}\sum_{\tau\in\Sh(\underbrace{1,\cdots,1}_{k\ \mathrm{times}}, n-k)}\lambda^{k-1}f(\underbrace{\dd\otimes\dots\otimes \dd}_{k\ \mathrm{times}}\otimes \Id^{\otimes n-k}\big)(x_1, \dots, x_{n})\\
	&=& \sum_{\sigma\in \Sh(k,n-k)}\sgn(\sigma)f_k(x_1, \cdots, \hat{x_{i_1}},\cdots, \hat{x_{i_k}},\cdots,x_n ,\dd(x_{i_1}),\cdots,\dd(x_{i_k}))-\dd(f(x_1,\cdots,x_n))\\
	&=&\delta^n(f).
\end{eqnarray*}

That is, we have
$$ l_1^{(s\mu, \dd)}(sf, g) =\big(-s\partial_\Lie^{n}(f), \partial_\DO^{n-1}(g)+\delta^n(f)\big)=-\partial^n_{\Diffl}(f,g).$$
\end{proof}


Let $(V,\dd_V)$ be a representation of the differential Lie algebra $(\frak\frakg,\dd_{\frakg})$ with weight $\lambda$.
Now we justify the cochain complex $(\rmC_{\Diffl}^*(\frakg, V),\partial_{\Diffl}^*)$ of the differential Lie algebra $(\frak\frakg,\dd_{\frakg})$ with coefficients in the representation $(V,\dd_V)$ introduced in Definition \ref{def: cochain complex for differential Lie algebras} and Proposition \ref{prop:delta}.

To this end,  we consider the trivial extension  $ \frakg \ltimes V$ of the differential Lie algebra $(\frakg,\dd_{\frakg})$ by the representation $(V,\dd_V)$ in Proposition~\ref{Prop: trivial extensions}, and we get the complex $( \rmC^*_{\Diffl}(\frakg \ltimes V, (\frakg \ltimes V)_\ad),\partial_{\Diffl}^*)$, by  Proposition \ref{cochaincomplexad}. The following result can be proved by direct inspection.

\begin{prop} \label{dlcohomology}
The cochain complex $(\rmC_{\Diffl}^*(\frakg, V),\partial_{\Diffl}^*)$ of the differential Lie algebra $(\frak\frakg,\dd_{\frakg})$ with coefficients in the representation $(V,\dd_V)$ introduced in Definition \ref{def: cochain complex for differential Lie algebras}   is a subcomplex of
	  $( \rmC^*_{\Diffl}(\frakg \ltimes V, (\frakg \ltimes V)_\ad),\partial_{\Diffl}^*)$.
\end{prop}

 \begin{proof} For vector spaces $U$ and $W$, we have  an isomorphism $$\wedge^p(U\oplus W)\simeq \wedge^p U\oplus (\oplus_{i=1}^{p-1} \wedge^iU\otimes \wedge^{p-i}W)\oplus \wedge^pW.$$
 Hence,
 there exists a natural injection
 from $$\rmC_{\Diffl}^n(\frakg, V)=\mathrm{Hom}( \wedge^n\frakg,V)\oplus \mathrm{Hom}( \wedge^{n-1}\frakg,V)$$ to $$\rmC^n_{\Diffl}(\frakg \ltimes V, (\frakg \ltimes V)_\ad)=\mathrm{Hom}( \wedge^n(\frakg \oplus  V),\frakg \oplus V)\oplus \mathrm{Hom}( \wedge^{n-1}(\frakg \oplus V),\frakg \oplus V).$$
 	Now it suffices to show that this map commutes with differentials.
	
 \end{proof}

\subsection{Homotopy  differential Lie algebras with weight}\

\label{Subsect: Homotopy  differential Lie algebras}

	In this section, we need to use   graded  Nijenhuis-Richardson brackets.
	
	Let $\frakg$ be a graded vector space and denote $\frakl=s\frakg$. Consider the graded vector space $\frakc (\frakg,\frakg):=\mathrm{Hom}(S(\frakl),\frakl)$. 

Recall that by Eq.~\eqref{Eq: exterior vs symmetric},  $S^n(s\frakg)$ is isomorphic to  $s^n\wedge^n\frakg$.
When $\frakg=\frakg^0=V$ is ungraded, then  $\Hom(S^n(sV),sV)$, which has degree $n-1$,  is isomorphic to $ \Hom(\wedge^n V,V)$ as vector spaces. This  justifies  why we have imposed degree $n-1$ on the latter in Subsection~\ref{Subsect: Linifnity structure for relative  differential Lie algebras}.

  The graded Nijenhuis-Richardson
	bracket $[-,-]_{\NR}$ on the graded vector space $\frakc (\frakg,\frakg)$ can be defined as follows:
	for $f \in \mathrm{Hom}(S^n(\frakl),\frakl)$ of degree $p$ and $g\in \mathrm{Hom}(S^m(\frakl),\frakl)$ of degree $q$,
	\begin{equation*}
		[f,g]_{\NR}:=f\bar\circ g-(-1)^{pq}g\bar\circ f,
	\end{equation*}
where $f\bar\circ g\in \mathrm{Hom}(S^{m+n-1}(\frakl),\frakl)$ of degree $p+q$ is defined by
\begin{equation*}
	\begin{aligned}
		f\bar\circ g(v_1,\dots,v_{m+n-1})
		&=\sum_{\sigma\in \Sh(m,n-1)}\varepsilon(\sigma)f(g(v_{\sigma(1)},\dots, v_{\sigma(m)}),v_{\sigma(m+1)},\dots, v_{\sigma(m+n-1)}), \end{aligned}
\end{equation*}
for $v_1, \dots, v_{m+n-1}\in \frakl$.
It is well known that $(\frakc (\frakg,\frakg),[~,~]_\NR)$ is a graded Lie algebra whose Maurer-Cartan elements correspond bijectively  to   $L_\infty[1]$-algebra structures  on $\frakl$.

 We consider the graded Lie algebra $$\frakL:=\frakc (\frakg\oplus\mathfrak{g},\frakg\oplus\mathfrak{g})$$ endowed with the  graded Nijenhuis-Richardson bracket $[~,~]_{\NR}$.

Let $$\frakm:=\mathrm{Hom}(\bar{S}(\frakl),\frakl)\ \mathrm{and}\
	\fraka:=\Hom(\bar{S}(\frakl),\frakl).$$
	Endow $\frakm$ with the graded Nijenhuis-Richardson bracket and $\fraka$ with the trivial bracket.
	With the same analysis in Subsection~\ref{Subsect:  Linfinty for differential Lie algebras}, we have:
	\begin{prop}\label{prop:graded Linfinity for absolute}
	Let $\frakg$ be a graded vector space.
		Then we get a generalised   V-datum $(\frakL, \frakm,\iota_\frakm, \fraka$, $\iota_\fraka, P,\Delta=0)$
		and  there is an $L_\infty$[1]-algebra $s \frakm\oplus\fraka$ given below:
		$$
		l_2(sf,sg)      =    (-1)^{|f|} s[f,g]_{\NR}, \ \
		l_2(sf,\xi)     =     [f, \xi]_{\NR}, $$
		and for $3\leq i\leq n+2$,
		$$l_i(sf,\xi_1,\cdots,\xi_{i-1})      =   \sum\limits_{\tau\in\Sh(m_{i-1}+1,\dots,m_1+1,n+2-i)}(-1)^{\sum\limits_{j=1}^{i-1}(|\xi_1|+\cdots+|\xi_{j-1}|)|\xi_j|}\lambda^{i-2} f  \big(\xi_{i-1}\otimes\dots\otimes\xi_1\otimes\Id^{\otimes n+2-i}\big)\tau^{-1},
		$$
		for homogeneous elements
		$f\in\Hom(S^{n+1}(\frakl),\frakl)\subseteq \frakm$, $g\in\Hom(S^{m+1}(\frakl),\frakl)\subseteq \frakm$, $\xi\in \Hom(S^{m+1}(\frakl),\frakl)\subseteq \fraka$, 	and $\xi_j\in\Hom(S^{m_j+1}(\frakl),\frakl)\subseteq \fraka$, $1\leq j\leq i-1$, and all others components vanish.
\end{prop}


%

\begin{defn} Let $\frakg$ be a graded vector space.
	A structure of \textbf{homotopy  differential Lie algebra with weight} $\lambda$ on $\frakg$ is defined to be a Maurer-Cartan element of the $L_\infty$[1]-algebra $s \frakm\oplus\fraka$ introduced in Proposition~\ref{prop:graded Linfinity for absolute}.
\end{defn}

\begin{theorem}Let $\frakg$ be a graded vector space and denote $\frakl=s\frakg$.
	A homotopy  differential Lie algebra with weight $\lambda\in\bfk $ on    $ \frakg$  is equivalent to the pair $(\mu=\{\mu_i\}_{i\ge 1},   D=\{D_i\}_{i\ge 1})$, where for each $i\ge 1$, $\mu_i: S^i(\frakl) \to \frakl$ is a degree $1$ map such that $(\frakg, \mu)$ is an $L_\infty[1]$-algebra and  for each $i\ge 1$, $D_i: S^i(\frakl) \to \frakl$ is of degree $0$ which form a homotopy differential operator of weight $\lambda$.  More precisely, for $n\ge 1$ and  $x_1, \dots, x_n\in \frakl$,	we have
	\begin{equation}\label{L-infty[1]}
		\sum_{i=1}^n\sum_{\sigma\in \Sh(i,n-i)}\varepsilon(\sigma)\mu_{n-i+1}(\mu_i(x_{\sigma(1)},\cdots,x_{\sigma(i)}),x_{\sigma(i+1)},\cdots,x_{\sigma(n)})=0,
	\end{equation}
and
\begin{equation}\label{homotopy diff operator}
	\begin{aligned}
		&\sum_{p\geq2}\sum_{t=p-1}^{n}\sum_{m_1+\cdots+m_{p-1}=t}\sum_{\substack{\sigma\in \mathrm{Sh}(m_{p-1},\cdots,m_1,n-t)\\\sigma|_t\in \mathrm{PSh}(m_{p-1},\cdots,m_1)}}\varepsilon(\sigma)\lambda^{p-2}\mu_{n-t+p-1}(D_{m_{p-1}}(x_{\sigma(1)},\cdots,x_{\sigma(m_{p-1})}),\cdots,\\
		&\qquad \cdots, D_{m_{1}}(x_{\sigma(m_2+\cdots+m_{p-1}+1)},\cdots,x_{\sigma(t)}),x_{\sigma(t+1)}\cdots,x_{\sigma(n)})\\
		&-\sum_{j=1}^n\sum_{\sigma\in \Sh(j,n-j)}\varepsilon(\sigma)D_{n-j+1}(\mu_j(x_{\sigma(1)},\cdots,x_{\sigma(j)}),x_{\sigma(j+1)},\cdots,x_{\sigma(n)})=0.
	\end{aligned}
\end{equation}

\end{theorem}
\begin{proof}
	Let  $\mu=\sum\limits_{i=1}^\infty\mu_i:\bar S(\frakl) \lon\frakl$, with   $\mu_i: S^i(\frakl)\lon\frakl$ and $D=\sum\limits_{i=1}^\infty D_i:\bar S(\frakl)\lon\frakl$, with   $D_i:  S^i(\frakl)\lon\frakl$.
	Then by Proposition \ref{Prop: sM+a is L_inf alg}, $(s\mu,D) \in\mathcal{MC}(s \mathfrak{M}\oplus\fraka)$ is equivalent to
	 $$[\mu,\mu]_{\NR}=0\  \mathrm{and}\  \sum\limits_{p=2}^\infty\frac{1}{(p-1)!}l_p(s\mu,\underbrace{D,\dots,D}_{(p-1)\ \mathrm{times}})=0.$$
	It is well known that, see for instance,    \cite[Theorem 4.2]{CC22}, that $[\mu,\mu]_{\NR}=0$ if and only if Eq.~\eqref{L-infty[1]} holds if and only if  $\frakg$ is an $L_\infty[1]$-algebra.	

	On the other hand, we have
	\begin{equation*}
		\begin{aligned}
			0=&	\sum_{p=2}^\infty\frac{1}{(p-1)!}l_p(s\mu,\underbrace{D,\dots,D}_{(p-1)\ \mathrm{times}})\\			=&\sum_{p=2}^\infty\frac{1}{(p-1)!}l_p(s\sum_{m=1}^{\infty}\mu_m,\sum_{m_1=1}^{\infty}D_{m_1},\dots,\sum_{m_{p-1}=1}^{\infty}D_{m_{p-1}})\\
			=& \sum_{p=2}^\infty\frac{1}{(p-1)!}\sum_{m, m_1,\dots, m_{p-1}=1}^\infty
			l_p(s\mu_m,D_{m_1},\dots,D_{m_{p-1}})\\
			=& \sum_{m=1}^\infty\sum_{j=1}^\infty[\mu_m,D_j]_\NR+\sum_{p=3}^\infty\frac{1}{(p-1)!}\sum_{m, m_1,\dots, m_{p-1}=1}^\infty
			\sum_{\tau\in\Sh(m_{p-1},\dots,m_1,m+1-p)}(-1)^{\sum\limits_{j=1}^{p-1}(|D_{m_1}|+\dots+| D_{m_{j-1}}|)|D_{m_j}|}
			\\	& \quad \quad \lambda^{p-2}\mu_m(D_{m_{p-1}}\otimes\dots\otimes D_{m_1}\otimes\Id^{\otimes m+1-p})\tau^{-1}\\
\end{aligned}
	\end{equation*}
		\begin{equation*}
		\begin{aligned}
			=& \sum_{m=1}^\infty\sum_{j=1}^\infty[\mu_m,D_j]_\NR\\
			&\quad+\sum_{p=3}^\infty\frac{1}{(p-1)!}\sum_{m, m_1,\dots, m_{p-1}=1}^\infty
			\sum_{\tau\in\Sh(m_{p-1},\dots,m_1,m+1-p)}
			\lambda^{p-2}\mu_m(D_{m_{p-1}}\otimes\dots\otimes D_{m_1}\otimes\Id^{\otimes m+1-p})\tau^{-1}, 
			\end{aligned}
	\end{equation*}
which is equivalent to Eq.~\eqref{homotopy diff operator}.
\end{proof}

\begin{exam} Since Eq.~\eqref{L-infty[1]} is the well known  $L_\infty[1]$-structure,
	we only need to focus on homotopy differential operators.

When $n=1$,  Eq.~\eqref{homotopy diff operator} gives
\begin{eqnarray}
	\mu_1 D_1=D_1\mu_1.\notag
\end{eqnarray}
which means that $D_1$ is a cochain map.

When $n=2$,  Eq.~\eqref{homotopy diff operator}   gives
\begin{eqnarray}
	D_1(\mu_2(x,y))-\mu_2(D_1(x),y)-\mu_2(x,D_1(y))-\lambda\mu_2(D_1(x),D_1(y))\notag\\=\mu_1(D_2(x,y))-D_2(\mu_1(x),y)-D_2(x,\mu_1(y)).\notag
\end{eqnarray}
which means that $D_1$ is a differential operator (with respect to the multiplication $\mu_2$), but only up to the homotopy given by  $D_2$.
\end{exam}

\bigskip

\noindent
{{\bf Acknowledgments.}  The first author was supported by National Natural Science Foundation of China (No. 11871071). The   fourth author was supported by the National Natural Science Foundation of China (No.  12071137), by  Key Laboratory of Ministry of Education, by  Shanghai Key Laboratory of PMMP  (No.   22DZ2229014),  and by Fundamental Research Funds for the Central Universities.



\end{document}